\documentclass[10pt]{amsart}
\usepackage{mathtools}
\usepackage{amscd}
\usepackage{amsmath}
\usepackage{amssymb}
\usepackage{amsfonts}
\usepackage{amsthm}
\usepackage{enumitem}
\usepackage{stmaryrd}

\input xy

\newcommand{\mcZ}{{\mathcal Z}}
\newcommand{\vgo}{\textrm{Vec}_G^\omega}
\newcommand{\vgou}{\textrm{Vec}_G^{\omega^u}}
\newcommand{\U}{{\operatorname{U}}}

\newcommand{\ra}{\rangle}

\newcommand\scalemath[2]{\scalebox{#1}{\mbox{\ensuremath{\displaystyle #2}}}}

\usepackage{tikz-cd}

\usepackage{hyperref}

\newcommand{\si}{\sigma_i}

\renewcommand{\Vec}{\text{Vec}}

\newcommand{\Z}{{\mathbb Z}}

\newcommand{\R}{{\mathbb R}}
\newcommand{\uno}{ \mathbf{1}}
\newcommand{\C}{{\mathbb{C} }}
\newcommand{\A}{{\mathcal{A} }}

\newcommand{\id}{\mbox{\rm id\,}}

\newcommand{\cB}{\mathcal{B}}

\newcommand{\Rep}{\operatorname{Rep}}

\theoremstyle{plain}

\numberwithin{equation}{section}

\newtheorem{theorem}{Theorem}[section]

\newtheorem{proposition}[theorem]{Proposition}

\theoremstyle{definition}
\newtheorem{definition}[theorem]{Definition}
\newtheorem{example}[theorem]{Example}

\theoremstyle{remark}

\theoremstyle{remark}

\newcounter{commentcounter}

\newcounter{todocounter}

\usepackage{braids}
\usetikzlibrary{decorations.pathreplacing, arrows}
\usetikzlibrary{decorations.markings}

\newcommand{\btikz}[1]{\begin{tikzpicture}[#1]}
\newcommand{\etikz}{\end{tikzpicture}}

\pgfdeclarelayer{top}
\pgfsetlayers{main,top}

\newcommand{\uotwostrandsigmaone}[6]{

\draw[#1] (1,0) node[black,below] {$#4$} to [out=90, in=-90] (2,1.5) node[black,above] {$#4$};

\draw[#2] (2,0) node[black, below] {$#5$} to [out=90,in=-30] (1.65,.65);
\draw[#3] (1.35,.85) to [out=150,in=-90] (1,1.5) node[black,above] {$#6$};

}

\newcommand{\poscross}[4]{\begin{scope}[decoration={markings, mark=at position .999 with {\arrow{>}}}]
\draw[thick, looseness=1.5,postaction={decorate}] (1,0) to [out=90, in=-90] (0,1.5);
\draw[white, line width=7, looseness=1.5] (0,0) to [out=90, in=-90] (1,1.5);
\draw[thick, looseness=1.5,postaction={decorate}] (0,0) to [out=90, in=-90] (1,1.5);
\draw (0,0) node[below] {$#1$};
\draw (1,0) node[below] {$#2$};
\draw (0,1.5) node[above] {$#3$};
\draw (1,1.5) node[above] {$#4$};
\end{scope}
}

\newcommand{\negcross}[4]{\begin{scope}[decoration={markings, mark=at position .999 with {\arrow{>}}}]
\draw[thick, looseness=1.5,postaction={decorate}] (0,0) to [out=90, in=-90] (1,1.5);
\draw[white, line width=7, looseness=1.5] (1,0) to [out=90, in=-90] (0,1.5);
\draw[thick, looseness=1.5,postaction={decorate}] (1,0) to [out=90, in=-90] (0,1.5);
\draw (0,0) node[below] {$#1$};
\draw (1,0) node[below] {$#2$};
\draw (0,1.5) node[above] {$#3$};
\draw (1,1.5) node[above] {$#4$};
\end{scope}
}

\author[Bonderson]{Parsa Bonderson}
\address{Microsoft Station Q, Santa Barbara, CA 93106-6105 USA}
\email{parsab@microsoft.com}
\author[Delaney]{Colleen Delaney}
\address{Department of Mathematics, UC Santa Barbara, Santa Barbara, CA }
\email{cdelaney@math.ucsb.edu}
\author[Galindo]{C\'esar Galindo}
\address{ Departamento de Matem\'aticas, Universidad de los Andes, Bogot\'a, Colombia}
\email{cn.galindo1116@uniandes.edu.co}
\author[Rowell]{Eric C. Rowell}
\address{Department of Mathematics, Texas A$\&$M University, College Station, TX}
\email{rowell@math.tamu.edu}
\author[Tran]{Alan Tran}
\address{Department of Physics, University of California, Santa Barbara, CA}
\email{adtran@physics.ucsb.edu}
\author[Wang]{Zhenghan Wang}
\address{Microsoft Station Q and Department of Mathematics\\University of California\\Santa Barbara, CA 93106}
\email{zhenghwa@microsoft.com,zhenghwa@math.ucsb.edu}

\begin{document}

\usetikzlibrary{math}

\title{On invariants of Modular categories beyond modular data}

\thanks{C.G. is partially supported by Faculty of Science of  Universidad de los Andes, Convocatoria 2018-2019 para la Financiaci\'on de Programas de Investigaci\'on, programa ''Simetr\'{i}a $T$ (inversi\'on temporal) en
categor\'{i}as de fusi\'{o}n y modulares''. E.C.R is partially supported by NSF grant DMS-1802145. Z.W. is partially supported by NSF grants DMS-1411212 and  FRG-1664351.}
\begin{abstract}
We study novel invariants of modular categories that are beyond the modular data, with an eye towards a simple set of complete invariants for modular categories.
Our focus is on the $W$-matrix---the quantum invariant of a colored framed Whitehead link from the associated TQFT of a modular category.  We prove that the $W$-matrix and the set of punctured $S$-matrices are strictly beyond the modular data $(S,T)$.  Whether or not the triple $(S,T,W)$ constitutes a complete invariant of modular categories remains an open question.
\end{abstract}


\date{\today}
\maketitle

\section{Introduction}

The classification of modular categories, called anyon models in physics if unitary, is an interesting and important problem both in mathematics and physics \cite{TuBook,ZW10,RW18}. Hence, a complete, yet simple set of invariants for modular categories is highly desirable.  It had been conjectured for some time that the modular data is such a candidate (e.g. \cite{RW18}), but recently counterexamples to this conjecture appeared \cite{Mignard-Shauenburg}.  This motivates the question: what additional invariants of modular categories would elevate modular data to a complete invariant? In this paper, we study link invariants as new invariants of modular categories and focus on one such invariant, called the $W$-matrix---a symmetric version of the quantum invariant of the Whitehead link.  The $W$-matrix is strictly beyond the modular data because it provides a new verification of the Mignard-Shauenburg counterexamples.  We do not know whether the modular data supplemented with $W$ is complete or not.

One complete set of data of modular categories comes from their skeletonization using $F$ and $R$ symbols (e.g. see \cite{DHW, Parsathesis}).  Unlike the modular data, the set of $F$ and $R$ symbols are not intrinsic because they depend on some gauge choices when their consistency pentagon and hexagon equations are solved. Geometric invariant theory can be used to show that for multiplicity-free modular categories a finite set of polynomial combinations of the $F$ and $R$ symbols gives rise to an intrinsic complete invariant, but the existence proof does not provide any practical formulas in general \cite{hagge15,MTthesis}.  However, intrinsic invariants of modular categories are easy to come by, and can be expressed as certain combinations of $F$ and $R$ symbols that are independent of gauge choices.  One example is $\sum_{\mu\nu}[F^{a,b,a}_{b}]_{b\mu\nu,b\nu\mu}$ for any fusion category (with $N_{ab}^{b}=N_{ba}^{b} \neq 0$).

A modular category is essentially the $2$-dimensional part of a $(2+1)D$-TQFT $(V,Z)$---called the modular functor $V$ (e.g. \cite{TuBook,RW18}).  
Since every modular category $\mathcal{B}$ leads to a $(2+1)D$-TQFT $(V,Z)$ \cite{TuBook}, every pair $(M, L)$, where $L$ is a colored framed oriented link in a $3$-manifold $M$, gives an invariant $Z(M,L)$ of modular categories.  The un-normalized modular data is the case in which $M$ is the $3$-sphere and $L$ is the Hopf link or $+1$-framed unknot, respectively.  Physically, it is almost a tautology that the set of all such invariants $Z(M,L)$ should determine a modular category uniquely.  Since each $3$-manifold $M$ can be obtained by surgery on a link in $S^3$, we may assume $M=S^3$ without loss of generality.  We speculate that a finite set $\{Z(S^3,L_i)\}$ of links $L_i\subset S^3$ would provide a complete invariant of modular categories.  In this paper, we search for knots and links with increasing complexity to see if their invariants go beyond the modular data.  We prove that all invariants of $2$-braid closures are determined by the modular data.  On the other hand, we find that the Whitehead link provides an invariant that is beyond the modular data.  We verify this using the Mignard-Shauenburg counterexamples \cite{Mignard-Shauenburg}.

It is almost certain that all modular functors satisfy the Grothendieck reconstruction principle \cite{luo99}.  If so, then the data of a modular category should be supported on punctured spheres up to $4$ punctures, the torus, and the once-punctured torus with consistency relations supported on the $5$-punctured sphere and twice-punctured torus.  Therefore, the first obvious place to look for invariants of modular categories beyond modular data would be the punctured $S$-matrices.  Unlike the $S$-matrix, the punctured $S$-matrices are generally not intrinsic. However, the trace of the diagonal components of the punctured $S$-matrices are intrinsic. Moreover, these intrinsic objects are closely related to the Whitehead link invariants, which we will focus on instead.

The representation category of a Drinfeld double $D^\omega(G)$ is a unitary modular category \cite{DPR90,AC92}.  Such modular categories probably comprise most of the modular categories because the number of such modular categories for a fixed rank grows faster than any polynomial of the rank \cite[Remark 4.5]{rankfinite16}.  The Mignard-Shauenburg counterexamples are also of the form $\Rep(D^\omega (G))$.  Therefore, it would be interesting to test on those modular categories whether or not the triple $(S,T,W)$ would be sufficient to classify them.

The contents of the paper are as follows.  In Sect. 2, we discuss some new invariants of modular categories.  In Sect. 3, we provide an exposition on computing operator invariants of colored braids from the twisted doubles $D^\omega(G)$. These representations of colored braid groups are not the topological anyonic representations that model anyon statistics, but rather are explicitly local and quasi-localizations of the anyonic representations.  In Sect. 4, we compute the $W$-matrix for the simplest Mignard-Shauenburg counterexample.  In Sect. 5, we comment on a systematic search of link invariants that go beyond modular data.

\section{Invariants of Modular Categories}

As explained in the introduction, the most obvious data that might go beyond the modular data would be the intrinsic information contained in the punctured S-matrices. This is, indeed, true, as our results will demonstrate. However, such data is difficult to access without the full $F$ and $R$ symbols of a modular category.  In principle, any link invariant could be equally proposed as a new invariant for modular categories. Since link invariants of two-strand braid closures are determined by the modular data, the Whitehead link is a good candidate because its minimum representative braid word has the shortest possible length among three-strand braid words representing oriented knots and links \cite{gittings}. More importantly, we focus on the Whitehead link due to the close relation between its colored link invariants and the punctured $S$-matrices.

In the following, we use both the language of tensor category theory and anyon model.  For a correspondence of terminologies, see e.g. \cite[Table 1]{RW18} or \cite{ZW10}.  We also use the convention that for an anyon $X$, $x$ denotes its isomorphism class or topological charge or anyon type.

\subsection{Punctured $S$-matrices and $W$-matrix}

\subsubsection{Punctured $S$-matrices}

The punctured $S$-matrix generalizes the modular $S$-matrix to that of a torus with a boundary carrying a topological charge (the modular S-matrix corresponding to the trivial charge). 
Given a modular category $\cB$ and a simple object $X$ of $\cB$, the associated $(2+1)D$-TQFT of $\cB$ leads to a projective representation $V(T^2_0;X)$ of the mapping class group of the colored one-hole torus $T^2_0$, where the hole is colored by the simple object $X$.  When a particular basis of $V(T^2_0;X)$ is chosen, an oriented mapping class that exchanges the meridian and longitude gives rise to a matrix, which will be called the punctured $S$-matrix $S^{(x)}=\left(S^{(x)}_{(a\mu)(b\nu)} \right)$, where $x$ denotes the isomorphism class of $X$.  It is obvious that the punctured $S$-matrix $S^{(x)}$ in general depends on the choices of the basis of $V(T^2_0;X)$.  A basis of $V(T^2_0;X)$ is chosen so that $S^{(x)}$ is the following in graphical calculus.

$$S^{(x)}_{(a\mu)(b\nu)}=\frac{1}{D\sqrt{d_x}}
\begin{tikzpicture}[baseline=0, thick,scale=.5, shift={(0,-4.8)}]

	
	\begin{scope}[decoration={markings, mark=at position 0.35 with {\arrow{<}}}]
		\draw[postaction={decorate}] (2, 5+1.25) arc (90:360:1.25cm);
		\draw (2, 5+1.25) arc (90:45:1.25cm);
		\draw (2+1.25, 5) arc (0:30:1.25cm);
		\draw (2-1.25,5) node[left] {\small $a$};
	\end{scope}
	
	\begin{scope}[decoration={markings, mark=at position 0.005 with {\arrow{>}}}]
		\draw[postaction={decorate}] (4-1.25, 5) arc (180:-90:1.25cm);
		\draw (4, 5-1.25) arc (270:225:1.25cm);
		\draw (4-1.25, 5) arc (180:210:1.25cm);
		\draw (4-1.25,5) node[left] {\small $b$};
	\end{scope}
	
	\begin{scope}[decoration={markings, mark=at position 0.5 with {\arrow{>}}}]
		\draw[black,looseness=1.3] (2,5-1.25) to [out=-90, in=-100] (4+1+0.95, 5-1); 
		\draw[black,looseness=1.6] (4,5+1.25) to [out=90, in=100] (4+1+1, 5+1.3)  ;
		\draw[black,looseness=0.4, postaction={decorate}] (4+1+0.95, 5-1) to [out=180-100,in=100-180] (4+1+1,5+1.3);
		\draw (4+2.1,5) node[right] {\small $x$};
		\draw (2.1,5-1.5) node[left] {\tiny $\mu$};
		\draw (4.1,5+1.4) node[left] {\tiny $\nu$};
	\end{scope}
\end{tikzpicture}
$$

It follows that the eigenvalues of each punctured $S$-matrix $S^{(x)}$ are intrinsic data of the modular category $\cB$, in particular the trace and determinant of $S^{(x)}$.  We use the graphical calculus for modular categories to compute below (see \cite{TuBook, ZW10, Parsathesis}).

\subsubsection{$W$-matrix}

The Whitehead link is a two-component link\footnote{The Whitehead link was first discovered by J.~Maxwell as an example of a non-trivial link with linking number=$0$ given by Gauss's formula.  It is called the Whitehead link because it was later used by J.H.C.~Whitehead to construct the Whitehead continuum---a contractible $3$-manifold that is not homeomorphic to ${\R}^3$.} that is not equivalent to its mirror.  When the distinction is important, the following is the one we refer to as the Whitehead link:


$$
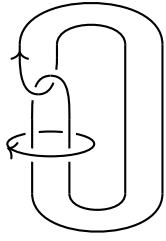
\begin{tikzpicture}[baseline=0, thick,scale=.33, shift={(0,-4.8)}]


\begin{scope}[decoration={markings, mark=at position 0.5 with {\arrow{>}}}, evaluate={
	\l = 1;
	\cut=5;
},]

\draw [looseness=\l](2.,6.75) to [out=180, in=90] (1.25,4.75);
\draw [looseness=\l,white,line width=\cut](1.5,6) to [out=180, in=-90] (0.75,7.5);
\draw [looseness=\l](1.5,6) to [out=180, in=-90] (0.75,7.5);

\draw [looseness=\l](1.5,6) to [out=0, in=-90] (2.25,7.5);

\draw [looseness=\l,white,line width=\cut](2,6.75) to [out=0, in=90 ] (2.75,4.75);
\draw [looseness=\l](2.,6.75) to [out=0, in=90] (2.75,4.75);

\draw (2.75,4.75) -- (2.75,2.0);
\draw (1.25,4.75) -- (1.25,2.0);
\draw (2.75,2.0) to [out=-90, in=-90] (5,2.0);
\draw (1.25,2.0) to [out=-90, in=-90] (6.5,2.0);

\draw[postaction=decorate] (0.75,7.5) -- (0.75,8);
\draw (2.25,7.5) -- (2.25,8);
\draw (2.25,8) to [out=90, in=90] (5,8);
\draw (0.75,8) to [out=90, in=90] (6.5,8);

\draw (5,2) -- (5,8);
\draw (6.5,2) -- (6.5,8);
\end{scope}

\begin{scope}[decoration={markings, mark=at position 0.56 with {\arrow{<}}}, evaluate={
	\l = 1.2;
	\cut=5;
},]
\draw[white, line width=\cut] (2,4) ellipse (1.75cm and 0.5cm);
\draw[postaction={decorate}] (2,4) ellipse (1.75cm and 0.5cm);
\draw[white, line width=\cut] (1.25,5) -- (1.25,4);
\draw[white, line width=\cut] (2.75,5) -- (2.75,4);
\draw (1.25,5) -- (1.25,4);
\draw (2.75,5) -- (2.75,4);
\end{scope}

\end{tikzpicture}
$$
The orientation of each link component is unimportant, as the four choices of orientations are equivalent under isotopy, as we will demonstrate.

Given a complete representative set $\Pi_{\cB}$ of simple objects of the modular category $\cB$,  $\widetilde{W}_{ab}$ is defined as the following colored oriented framed link invariant:


$$
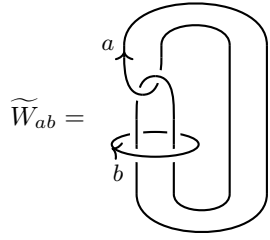
\widetilde{W}_{ab}= 
\begin{tikzpicture}[baseline=0, thick,scale=.33, shift={(0,-4.8)}]


\begin{scope}[decoration={markings, mark=at position 0.5 with {\arrow{>}}}, evaluate={
	\l = 1;
	\cut=5;
},]

\draw [looseness=\l](2.,6.75) to [out=180, in=90] (1.25,4.75);
\draw [looseness=\l,white,line width=\cut](1.5,6) to [out=180, in=-90] (0.75,7.5);
\draw [looseness=\l](1.5,6) to [out=180, in=-90] (0.75,7.5);

\draw [looseness=\l](1.5,6) to [out=0, in=-90] (2.25,7.5);

\draw [looseness=\l,white,line width=\cut](2,6.75) to [out=0, in=90 ] (2.75,4.75);
\draw [looseness=\l](2.,6.75) to [out=0, in=90] (2.75,4.75);

\draw (2.75,4.75) -- (2.75,2.0);
\draw (1.25,4.75) -- (1.25,2.0);
\draw (2.75,2.0) to [out=-90, in=-90] (5,2.0);
\draw (1.25,2.0) to [out=-90, in=-90] (6.5,2.0);

\draw[postaction=decorate] (0.75,7.5) -- (0.75,8);
\draw (0.75,8) node[left] {\small $a$};
\draw (2.25,7.5) -- (2.25,8);
\draw (2.25,8) to [out=90, in=90] (5,8);
\draw (0.75,8) to [out=90, in=90] (6.5,8);

\draw (5,2) -- (5,8);
\draw (6.5,2) -- (6.5,8);
\end{scope}

\begin{scope}[decoration={markings, mark=at position 0.56 with {\arrow{<}}}, evaluate={
	\l = 1.2;
	\cut=5;
},]
\draw[white, line width=\cut] (2,4) ellipse (1.75cm and 0.5cm);
\draw[postaction={decorate}] (2,4) ellipse (1.75cm and 0.5cm);
\draw (0.5,3.6) node[below] {\small $b$};
\draw[white, line width=\cut] (1.25,5) -- (1.25,4);
\draw[white, line width=\cut] (2.75,5) -- (2.75,4);
\draw (1.25,5) -- (1.25,4);
\draw (2.75,5) -- (2.75,4);
\end{scope}

\end{tikzpicture}
$$

Even though the Whitehead link is symmetric as a link, it is not symmetric as an oriented framed link.  Hence, the matrix $\widetilde{W}=(\widetilde{W}_{ab})$ is not symmetric as type-I Reidemeister moves are used in the process of exchanging the two components.  We can, however, define a symmetric matrix $W=(W_{ab})$ from $\widetilde{W}_{ab}$ using twist factors as follows.

\begin{definition}
Let $W_{ab}=\frac{\theta_a}{\theta_b}\widetilde{W}_{ab}$ for any pair $a,b\in \Pi_{\cB}$.  Then the matrix $W=(W_{ab})_{a,b \in \Pi_{\cB}}$ will be called the $W$-matrix of the modular category $\cB$.
\end{definition}

\begin{proposition}\label{Wmatrixproperties}
The $W$-matrix is symmetric and determined by the punctured $S$-matrices together with the modular data.

\begin{enumerate}
    \item $$\theta_a^2\widetilde{W}_{ax}=\theta_x^2\widetilde{W}_{x{\bar{a}}}$$.
    \item $$\widetilde{W}_{ax}=\widetilde{W}_{a\bar{x}}.$$
\item $$\sum^{N^z_{a\bar{a}}}_{\mu=1}S^{(z)}_{(a\mu)(a\mu)}=\frac{d_a}{\theta_a D^2}\sum_x S_{zx} \theta_{x} W_{ax}.$$
\item $$W_{ab}=\frac{\theta_a D^2}{\theta_b d_a}\sum_{x,\mu} S_{bx}^* S^{(x)}_{(a\mu)(a\mu)}.$$
\end{enumerate}

\end{proposition}

\begin{proof}
(1):
\begin{align*}
\widetilde{W}_{ax}
&= 
\begin{tikzpicture}[baseline=0, thick,scale=.33, shift={(0,-4.8)}]
%
%
\begin{scope}[decoration={markings, mark=at position 0.5 with {\arrow{>}}}, evaluate={
	\l = 1;
	\cut=5;
},]
\draw [looseness=\l](2.,6.75) to [out=180, in=90] (1.25,4.75);
\draw [looseness=\l,white,line width=\cut](1.5,6) to [out=180, in=-90] (0.75,7.5);
\draw [looseness=\l](1.5,6) to [out=180, in=-90] (0.75,7.5);
\draw [looseness=\l](1.5,6) to [out=0, in=-90] (2.25,7.5);
\draw [looseness=\l,white,line width=\cut](2,6.75) to [out=0, in=90 ] (2.75,4.75);
\draw [looseness=\l](2.,6.75) to [out=0, in=90] (2.75,4.75);
\draw (2.75,4.75) -- (2.75,2.0);
\draw (1.25,4.75) -- (1.25,2.0);
\draw (2.75,2.0) to [out=-90, in=-90] (5,2.0);
\draw (1.25,2.0) to [out=-90, in=-90] (6.5,2.0);
\draw[postaction=decorate] (0.75,7.5) -- (0.75,8);
\draw (0.75,8) node[left] {\small $a$};
\draw (2.25,7.5) -- (2.25,8);
\draw (2.25,8) to [out=90, in=90] (5,8);
\draw (0.75,8) to [out=90, in=90] (6.5,8);
\draw (5,2) -- (5,8);
\draw (6.5,2) -- (6.5,8);
\end{scope}
\begin{scope}[decoration={markings, mark=at position 0.56 with {\arrow{<}}}, evaluate={
	\l = 1.2;
	\cut=5;
},]
\draw[white, line width=\cut] (2,4) ellipse (1.75cm and 0.5cm);
\draw[postaction={decorate}] (2,4) ellipse (1.75cm and 0.5cm);
\draw (0.5,3.7) node[below] {\small $x$};
\draw[white, line width=\cut] (1.25,5) -- (1.25,4);
\draw[white, line width=\cut] (2.75,5) -- (2.75,4);
\draw (1.25,5) -- (1.25,4);
\draw (2.75,5) -- (2.75,4);
\end{scope}
\end{tikzpicture}
=
\begin{tikzpicture}[baseline=0, thick,scale=.33, shift={(0,-4.8)}]
%
%
\begin{scope}[decoration={markings, mark=at position 0.5 with {\arrow{>}}}, evaluate={
	\l = 1;
	\cut=4;
},]
\draw [looseness=\l](2.,6.75) to [out=180, in=90] (1.25,4.75);
\draw [looseness=\l,white,line width=\cut](1.5,6) to [out=180, in=-90] (0.75,7.5);
\draw [looseness=\l](1.5,6) to [out=180, in=-90] (0.75,7.5);
\draw [looseness=\l](1.5,6) to [out=0, in=-90] (2.25,7.5);
\draw [looseness=\l,white,line width=\cut](2,6.75) to [out=0, in=90 ] (2.75,4.75);
\draw [looseness=\l](2.,6.75) to [out=0, in=90] (2.75,4.75);
\draw (2.75,4.75) -- (2.75,2.0);
\draw (1.25,4.75) -- (1.25,2.0);
\draw (2.75,2.0) to [out=-90, in=-90] (5,2.0);
\draw (1.25,2.0) to [out=-90, in=-90] (6.5,2.0);
\draw[postaction=decorate] (0.75,7.5) -- (0.75,8);
\draw (0.75,8) node[left] {\small $a$};
\draw (2.25,7.5) -- (2.25,8);
\draw (2.25,8) to [out=90, in=90] (5,8);
\draw (0.75,8) to [out=90, in=90] (6.5,8);
\draw (5,2) -- (5,8);
\draw (6.5,2) -- (6.5,8);
\end{scope}
\begin{scope}[decoration={markings, mark=at position 0.33 with {\arrow{>}}}, evaluate={
	\l = 1.2;
	\cut=4;
},]
\draw[white, line width = \cut, looseness=0.8] (1.7, 6.3) to [out=90, in=180] (3.0,8);
\draw[black,looseness=0.8] (1.7, 6.3) to [out=90, in=180] (3.0,8);
\draw [white, line width = \cut] (2.25,7.40) --  (2.25,8);
\draw [black](2.25,7.2) -- (2.25,8);
\draw[black,looseness=0.8] (3.0,8) to [out=0, in=90] (3.5, 6.4) ;
\draw[white, looseness=0.8, line width = \cut]  (3.5, 6.4)to [out=-90, in=0] (2.75,3.0) ;
\draw[white,line width=\cut,looseness=0.8]  (3.5, 6.4)to [out=-90, in=0] (2.7,3.0) ;
\draw[black,looseness=0.8]  (3.5, 6.4)to [out=-90, in=0] (2.75,3.0) ;
\draw[white, looseness=0.8, line width = \cut](2.75,3.0) to [out=180, in=-90] (1.7,5.8);
\draw[black,looseness=0.8,postaction=decorate](2.75,3.0) to [out=180, in=-90] (1.7,5.8);
\draw (2.75,2.8) node[left] {\small $x$};
\end{scope}
\end{tikzpicture}
\\
&=
\theta_a^{-1}
\begin{tikzpicture}[baseline=0, thick,scale=.33, shift={(0,-4.8)}]
%
%
\begin{scope}[decoration={markings, mark=at position 0.99 with {\arrow{>}}}, evaluate={
	\l = 1;
	\cut=5;
},]
\draw (5,8) to [out=0, in=90] (7,5);
\draw [postaction=decorate](5,2) to [out=180, in=-90] (3,5);
\draw (2.35,4.8) node[] {\small $x$};
\draw [white, line width=\cut] (1,1) to [out=0, in=180] (9,9);
\draw (1,1) to [out=0, in=180] (9,9);
\draw[white, line width=\cut] (1,9) to [out=0, in=180] (9,1);
\draw (1,9) to [out=0, in=180] (9,1);
\end{scope}
\begin{scope}[decoration={markings, mark=at position 0.5 with {\arrow{<}}}, evaluate={
	\l = 1;
	\cut=5;
},]
\draw [looseness=0.5,postaction=decorate](1,9) to [out=180, in=180] (1,1);
\draw (-0.1,5) node[right] {\small $a$};
\draw [looseness=0.5](9,1) to [out=0, in=0] (9,9);
\draw [white, line width=\cut] (5,8) to [out=-180, in=90] (3,5);
\draw (5,8) to [out=-180, in=90] (3,5);
\draw [white, line width=\cut] (5,2) to [out=0, in=-90] (7,5);
\draw (5,2) to [out=0, in=-90] (7,5);
\end{scope}
\end{tikzpicture}
=
\theta_a^{-1}\theta_x
\begin{tikzpicture}[baseline=0, thick,scale=.33, shift={(0,-4.8)}]
%
\draw (0.5,8) to [out=90, in=90]  (7.5,7.5);
\begin{scope}[decoration={markings, mark=at position 0.99 with {\arrow{>}}}, evaluate={
	\l = 1;
	\cut=5;
},]
\draw [looseness=0.4](0.5, 2.0) to [out=90,in=180] (3.5, 5.5);
\draw [white, line width=\cut] (2,2) to [out=0, in=180] (8,8);
\draw (2,2) to [out=0, in=180] (8,8);
\draw[white, line width=\cut] (2,8) to [out=0, in=180] (8,2);
\draw(2,8) to [out=0, in=180] (8,2);
\end{scope}
\begin{scope}[decoration={markings, mark=at position 0.65 with {\arrow{<}}}, evaluate={
	\l = 1;
	\cut=5;
},]
\draw [looseness=0.5, white, line width=\cut](2,8) to [out=180, in=180] (2,2);
\draw [looseness=0.5,postaction=decorate](2,8) to [out=180, in=180] (2,2);
\draw [looseness=0.5](8,2) to [out=0, in=0] (8,8);
\end{scope}
\begin{scope}[decoration={markings, mark=at position 0.1 with {\arrow{<}}}, evaluate={
	\l = 1;
	\cut=5;
},]
\draw [looseness=0.4,white, line width=\cut](0.5, 8.0) to [out=-90,in=180] (3.5, 4.5);
\draw [looseness=0.4, postaction=decorate](0.5, 8.0) to [out=-90,in=180] (3.5, 4.5);
\draw (3.5, 4.5) to [out=0,in=0] (3.5, 5.5);
\draw [white, line width=\cut]((0.5,2) to [out=-90, in=-90]  (7.5,2.5);
\draw (0.5,2) to [out=-90, in=-90]  (7.5,2.5);
\draw (7.5,2.5) -- (7.5, 7.5);
\end{scope}
\draw (1.0,4) node[left] {\small $a$};
\draw (0.7,7.5) node[left] {\small $x$};
\end{tikzpicture}
\\
&=
\theta_a^{-1}\theta_x
\begin{tikzpicture}[baseline=0, thick,scale=.33, shift={(0,-4.8)}]
%
\draw [looseness=0.7](1,6.5) to [out=0,in=90] (2.0,5);
\draw [looseness=1.1](0,6.5) to [out=90, in=90] (7,6.5);
\begin{scope}[decoration={markings, mark=at position 0.20 with {\arrow{<}}}, evaluate={
	\l = 1;
	\cut=5;
},]
\draw[white, line width = \cut] (2,3) to [out=0, in=180] (8,7);
\draw [postaction=decorate](2,3) to [out=0, in=180] (8,7);
\draw[white, line width=\cut] (2,7) to [out=0, in=180] (8,3);
\draw(2,7) to [out=0, in=180] (8,3);
\end{scope}
\begin{scope}[decoration={markings, mark=at position 0.65 with {\arrow{<}}}, evaluate={
	\l = 1;
	\cut=5;
},]
\draw [looseness=0.5, white, line width=\cut](2,7) to [out=180, in=180] (2,3);
\draw [looseness=0.5](2,7) to [out=180, in=180] (2,3);
\draw [looseness=0.5](8,3) to [out=0, in=0] (8,7);
\end{scope}
\begin{scope}[decoration={markings, mark=at position 0.00 with {\arrow{<}}}, evaluate={
	\l = 1;
	\cut=5;
},]
\draw [white, line width=\cut, looseness=0.7](1,3.5) to [out=0,in=-90] (2.0,5);
\draw [looseness=0.7](1,3.5) to [out=0,in=-90] (2.0,5);
\draw [looseness=0.5](1,6.5) to [out=180,in=90] (0,3.5);
\draw [white, line width=\cut, looseness=0.5](1,3.5) to [out=180,in=-90] (0,6.5);
\draw [looseness=0.5](1,3.5) to [out=180,in=-90] (0,6.5);
\draw [white, line width=\cut, looseness=1.1](0,3.5) to [out=-90,in=-90] (7,3.5);
\draw [postaction=decorate, looseness=1.1](0,3.5) to [out=-90,in=-90] (7,3.5);
\draw (7,3.5) --(7,6.5);
\end{scope}
\draw (4,3) node[] {\small $a$};
\draw (0,3.5) node[left] {\small $x$};
\end{tikzpicture}
=
\theta_a^{-2}\theta_x
\begin{tikzpicture}[baseline=0, thick,scale=.33, shift={(0,-4.8)}]
%
\draw [looseness=0.7](1,6.5) to [out=0,in=90] (2.0,5);
\draw [looseness=1.1](0,6.5) to [out=90, in=90] (7,6.5);
\begin{scope}[decoration={markings, mark=at position 0.20 with {\arrow{<}}}, evaluate={
	\l = 1;
	\cut=5;
},]
\draw[white, line width = \cut] (2,3) to [out=0, in=180] (8,7);
\draw [postaction=decorate](2,3) to [out=0, in=180] (8,7);
\draw[white, line width=\cut] (2,7) to [out=0, in=180] (8,3);
\draw(2,7) to [out=0, in=180] (8,3);
\end{scope}
\begin{scope}[decoration={markings, mark=at position 0.65 with {\arrow{<}}}, evaluate={
	\l = 1;
	\cut=5;
},]
\draw [looseness=0.5](8,3) to [out=0, in=0] (8,7);
\draw [black,looseness=0.5](2,3) to [out=180, in=0] (1,5.5);
\draw [white,line width=\cut,looseness=0.5](2,7) to [out=180, in=0] (1,4.5);
\draw [black,looseness=0.5](2,7) to [out=180, in=0] (1,4.5);
\draw [black,looseness=0.5](1,5.5) to [out=180, in=180] (1,4.5);
\end{scope}
\begin{scope}[decoration={markings, mark=at position 0.00 with {\arrow{<}}}, evaluate={
	\l = 1;
	\cut=5;
},]
\draw [white, line width=\cut, looseness=0.7](1,3.5) to [out=0,in=-90] (2.0,5);
\draw [looseness=0.7](1,3.5) to [out=0,in=-90] (2.0,5);
\draw [looseness=0.5](1,6.5) to [out=180,in=90] (0,3.5);
\draw [white, line width=\cut, looseness=0.5](1,3.5) to [out=180,in=-90] (0,6.5);
\draw [looseness=0.5](1,3.5) to [out=180,in=-90] (0,6.5);
\draw [white, line width=\cut, looseness=1.1](0,3.5) to [out=-90,in=-90] (7,3.5);
\draw [postaction=decorate, looseness=1.1](0,3.5) to [out=-90,in=-90] (7,3.5);
\draw (7,3.5) --(7,6.5);
\end{scope}
\draw (4,3) node[] {\small $a$};
\draw (0,3.5) node[left] {\small $x$};
\end{tikzpicture}
\\
&=
\theta_a^{-2}\theta_x
\begin{tikzpicture}[baseline=0, thick,scale=.33, shift={(-2,-4.8)}]
%
\draw (8.5,5) -- (8.5, 8);
\draw[looseness=1.0] (3,7) to [out=90,in=90] (8.5,8);
\draw [white, line width=5,looseness=0.8](5,5) to [out=90, in=90] (10,5);
\draw [looseness=0.8](5,5) to [out=90, in=90] (10,5);
\begin{scope}[decoration={markings, mark=at position 0.0 with {\arrow{<}}}, evaluate={
	\l = 1;
	\cut=5;
},]
\draw [looseness=0.5](3,3) to [out=90, in =180] (5,7);
\draw [white, line width=5, looseness=0.5](3,7) to [out=-90, in =180] (5,3);
\draw [postaction=decorate,looseness=0.5](3,7) to [out=-90, in =180] (5,3);
\draw [white, line width=\cut,looseness=1.2] (5,3) to [out=0, in=0] (5,7);
\draw [looseness=1.2] (5,3) to [out=0, in=0] (5,7);
\draw[white, line width = \cut,looseness=0.8] (5,5) to [out=-90, in=-90] (10,5);
\end{scope}
\begin{scope}[decoration={markings, mark=at position 0.99 with {\arrow{>}}}, evaluate={
	\l = 1;
	\cut=5;
},]
\draw[postaction = decorate, looseness=0.8] (5,5) to [out=-90, in=-90] (10,5);
\draw[white, line width=\cut,looseness=1.0] (3,3) to [out=-90,in=-90] (8.5,2);
\draw[looseness=1.0] (3,3) to [out=-90,in=-90] (8.5,2);
\draw[white,line width=\cut] (8.5,2) -- (8.5, 5);
\draw (8.5,2) -- (8.5, 5);
\end{scope}
\draw (10,5) node[right] {\small $a$};
\draw (3,7) node[left] {\small $x$};
\end{tikzpicture}
=
\theta_a^{-2}\theta_x^2
\begin{tikzpicture}[baseline=0, thick,scale=.33, shift={(-2,-4.8)}]
%
\draw (8.5,5) -- (8.5, 8);
\draw[looseness=1.0] (3,7) to [out=90,in=90] (8.5,8);
\draw [white, line width=5,looseness=0.8](5,5) to [out=90, in=90] (10,5);
\draw [looseness=0.8](5,5) to [out=90, in=90] (10,5);
\begin{scope}[decoration={markings, mark=at position 0.0 with {\arrow{<}}}, evaluate={
	\l = 1;
	\cut=5;
},]
\draw [looseness=0.5](3,3) to [out=90, in =180] (5,7);
\draw [white, line width=5, looseness=0.5](3,7) to [out=-90, in =180] (5,3);
\draw [postaction=decorate,looseness=0.5](3,7) to [out=-90, in =180] (5,3);
\draw [black,looseness=1.2] (5,3) to [out=0, in=180] (7,5.5);
\draw [white,line width=\cut,looseness=1.2] (5,7) to [out=0, in=180] (7,4.5);
\draw [black,looseness=1.2] (5,7) to [out=0, in=180] (7,4.5);
\draw [looseness=1.2] (7,4.5) to [out=0,in=0] (7,5.5);
\draw[white, line width = \cut,looseness=0.8] (5,5) to [out=-90, in=-90] (10,5);
\end{scope}
\begin{scope}[decoration={markings, mark=at position 0.99 with {\arrow{>}}}, evaluate={
	\l = 1;
	\cut=5;
},]
\draw[postaction = decorate, looseness=0.8] (5,5) to [out=-90, in=-90] (10,5);
\draw[white, line width=\cut,looseness=1.0] (3,3) to [out=-90,in=-90] (8.5,2);
\draw[looseness=1.0] (3,3) to [out=-90,in=-90] (8.5,2);
\draw[white,line width=\cut] (8.5,2) -- (8.5, 5);
\draw (8.5,2) -- (8.5, 5);
\end{scope}
\draw (10,5) node[right] {\small $a$};
\draw (3,7) node[left] {\small $x$};
\end{tikzpicture}
\\
&=
\theta_a^{-2}\theta_x^{2}
\begin{tikzpicture}[baseline=0, thick,scale=.33, shift={(0,-4.8)}]
%
\draw (0.75,8) to [out=90,in=90] (7,8);
\begin{scope}[decoration={markings, mark=at position 0.5 with {\arrow{>}}}, evaluate={
	\l = 1;
	\cut=5;
},]
\draw [looseness=0.8](5.5,6.75) to [out=-90,in=180] (7,5.0);
\draw [looseness=\l](2.,6.75) to [out=180, in=90] (1.25,4.75);
\draw [looseness=\l,white,line width=\cut](1.5,6) to [out=180, in=-90] (0.75,7.5);
\draw [looseness=\l](1.5,6) to [out=180, in=-90] (0.75,7.5);
\draw [looseness=\l,](1.5,6) to [out=0, in=180] (5,9);
\draw [looseness=\l,white,line width=\cut](2,6.75) to [out=0, in=180 ] (5,4.5);
\draw [looseness=\l,](2.,6.75) to [out=0, in=180] (5,4.5);
\draw[white, line width=\cut] (5,9) to [out=0, in =0] (5,4.5);
\draw (5,9) to [out=0, in =0] (5,4.5);
\draw[postaction=decorate] (0.75,7.5) -- (0.75,8);
\draw[white, line width=\cut,looseness=0.8](5.5,6.75) to [out=90,in=180] (7,8.5);
\draw [looseness=0.8](5.5,6.75) to [out=90,in=180] (7,8.5);
\end{scope}
\begin{scope}[decoration={markings, mark=at position 0.5 with {\arrow{>}}}, evaluate={
	\l = 1;
	\cut=5;
},]
\draw [looseness=1.2,postaction=decorate](7,5) to [out=0,in=0] (7,8.5);
\draw (1.25,4.75) to [out=-90,in=-90] (7,4.75);
\draw[white, line width = \cut] (7,4.75) -- (7,8);
\draw (7,4.75) -- (7,8);
\end{scope}
\draw (0.75,8) node[left] {\small $x$};
\draw (8.2,6.6) node[right] {\small $a$};
\end{tikzpicture}
= 
\theta_a^{-2}\theta_x^2
\begin{tikzpicture}[baseline=0, thick,scale=.33, shift={(0,-4.8)}]
%
%
\begin{scope}[decoration={markings, mark=at position 0.5 with {\arrow{>}}}, evaluate={
	\l = 1;
	\cut=5;
},]
\draw [looseness=\l](2.,6.75) to [out=180, in=90] (1.25,4.75);
\draw [looseness=\l,white,line width=\cut](1.5,6) to [out=180, in=-90] (0.75,7.5);
\draw [looseness=\l](1.5,6) to [out=180, in=-90] (0.75,7.5);
\draw [looseness=\l](1.5,6) to [out=0, in=-90] (2.25,7.5);
\draw [looseness=\l,white,line width=\cut](2,6.75) to [out=0, in=90 ] (2.75,4.75);
\draw [looseness=\l](2.,6.75) to [out=0, in=90] (2.75,4.75);
\draw (2.75,4.75) -- (2.75,2.0);
\draw (1.25,4.75) -- (1.25,2.0);
\draw (2.75,2.0) to [out=-90, in=-90] (5,2.0);
\draw (1.25,2.0) to [out=-90, in=-90] (6.5,2.0);
\draw[postaction=decorate] (0.75,7.5) -- (0.75,8);
\draw (0.75,8) node[left] {\small $x$};
\draw (2.25,7.5) -- (2.25,8);
\draw (2.25,8) to [out=90, in=90] (5,8);
\draw (0.75,8) to [out=90, in=90] (6.5,8);
\draw (5,2) -- (5,8);
\draw (6.5,2) -- (6.5,8);
\end{scope}
\begin{scope}[decoration={markings, mark=at position 0.6 with {\arrow{>}}}, evaluate={
	\l = 1.2;
	\cut=5;
},]
\draw[white, line width=\cut] (2,4) ellipse (1.75cm and 0.5cm);
\draw[postaction={decorate}] (2,4) ellipse (1.75cm and 0.5cm);
\draw (0.5,3.6) node[below] {\small $a$};
\draw[white, line width=\cut] (1.25,5) -- (1.25,4);
\draw[white, line width=\cut] (2.75,5) -- (2.75,4);
\draw (1.25,5) -- (1.25,4);
\draw (2.75,5) -- (2.75,4);
\end{scope}
\end{tikzpicture}
\\
&= 
\theta_a^{-2}\theta_x^2 \widetilde{W}_{x\bar{a}}
\end{align*}

(2):
	$$
	\widetilde{W}_{a x}=
\begin{tikzpicture}[baseline=0, thick,scale=.33, shift={(0,-4.8)}]
%
%
\begin{scope}[decoration={markings, mark=at position 0.5 with {\arrow{>}}}, evaluate={
	\l = 1;
	\cut=5;
},]
\draw [looseness=\l](2.,6.75) to [out=180, in=90] (1.25,4.75);
\draw [looseness=\l,white,line width=\cut](1.5,6) to [out=180, in=-90] (0.75,7.5);
\draw [looseness=\l](1.5,6) to [out=180, in=-90] (0.75,7.5);
\draw [looseness=\l](1.5,6) to [out=0, in=-90] (2.25,7.5);
\draw [looseness=\l,white,line width=\cut](2,6.75) to [out=0, in=90 ] (2.75,4.75);
\draw [looseness=\l](2.,6.75) to [out=0, in=90] (2.75,4.75);
\draw (2.75,4.75) -- (2.75,2.0);
\draw (1.25,4.75) -- (1.25,2.0);
\draw (2.75,2.0) to [out=-90, in=-90] (5,2.0);
\draw (1.25,2.0) to [out=-90, in=-90] (6.5,2.0);
\draw[postaction=decorate] (0.75,7.5) -- (0.75,8);
\draw (0.75,7.6) node[left] {\small $a$};
\draw (2.25,7.5) -- (2.25,8);
\draw (2.25,8) to [out=90, in=90] (5,8);
\draw (0.75,8) to [out=90, in=90] (6.5,8);
\draw (5,2) -- (5,8);
\draw (6.5,2) -- (6.5,8);
\end{scope}
\begin{scope}[decoration={markings, mark=at position 0.56 with {\arrow{<}}}, evaluate={
	\l = 1.2;
	\cut=5;
},]
\draw[white, line width=\cut] (2,4) ellipse (1.75cm and 0.5cm);
\draw[postaction={decorate}] (2,4) ellipse (1.75cm and 0.5cm);
\draw (0.5,3.6) node[below] {\small $x$};
\end{scope}
\begin{scope}[decoration={markings, mark=at position 0.0 with {\arrow{<}}}, evaluate={
	\l = 1.2;
	\cut=5;
},]
\draw[white, line width=\cut] (1.25,5) -- (1.25,4);
\draw[white, line width=\cut] (2.75,5) -- (2.75,4);
\draw[black, postaction=decorate] (1.25,5) -- (1.25,4);
\draw (1.25,5) node[left] {\small $a$};
\draw (2.75,5) -- (2.75,4);
\end{scope}
\end{tikzpicture}
=
\begin{tikzpicture}[baseline=0, thick,scale=.33, shift={(0,-4.8)}]
%
%
\begin{scope}[decoration={markings, mark=at position 0.5 with {\arrow{>}}}, evaluate={
	\l = 1;
	\cut=5;
},]
\draw [looseness=\l](2.,6.75) to [out=180, in=90] (1.25,4.75);
\draw [looseness=\l,white,line width=\cut](1.5,6) to [out=180, in=-90] (0.75,7.5);
\draw [looseness=\l](1.5,6) to [out=180, in=-90] (0.75,7.5);
\draw [looseness=\l](1.5,6) to [out=0, in=-90] (2.25,7.5);
\draw [looseness=\l,white,line width=\cut](2,6.75) to [out=0, in=90 ] (2.75,4.75);
\draw [looseness=\l](2.,6.75) to [out=0, in=90] (2.75,4.75);
\draw (2.75,4.75) -- (2.75,2.0);
\draw (1.25,4.75) -- (1.25,2.0);
\draw[] (2.75,2.0) to [out=-90, in=-90] (2.75-5.75,2.0);
\draw[] (1.25,2.0) to [out=-90, in=-90] (1.25-2.75,2.0);
\draw[postaction=decorate] (0.75,7.5) -- (0.75,8);
\draw (0.75,7.6) node[left] {\small $a$};
\draw (2.25,7.5) -- (2.25,8);
\draw[] (2.25,8) to [out=90, in=90] (2.25-5.25,8);
\draw[] (0.75,8) to [out=90, in=90] (0.75-2.25,8);
\draw (-3,2) -- (-3,8);
\draw (-1.5,2) -- (-1.5,8);
\end{scope}
\begin{scope}[decoration={markings, mark=at position 0.56 with {\arrow{<}}}, evaluate={
	\l = 1.2;
	\cut=5;
},]
\draw[white, line width=\cut] (-2.25,4) ellipse (1.75cm and 0.5cm);
\draw[postaction={decorate}] (-2.25,4) ellipse (1.75cm and 0.5cm);
\draw (-3.7,3.6) node[below] {\small $x$};
\end{scope}
\begin{scope}[decoration={markings, mark=at position 0.0 with {\arrow{<}}}, evaluate={
	\l = 1.2;
	\cut=5;
},]
\draw[white, line width=\cut] (-1.5,5) -- (-1.5,4);
\draw[white, line width=\cut] (-3,5) -- (-3,4);
\draw[] (-1.5,5) -- (-1.5,4);
\draw[] (-3,5) -- (-3,4);
\draw[black, postaction=decorate] (1.25,5) -- (1.25,4);
\draw (1.25,5) node[left] {\small $a$};
\draw (2.75,5) -- (2.75,4);
\end{scope}
\end{tikzpicture}
=
\begin{tikzpicture}[baseline=0, thick,scale=.33, shift={(0,-4.8)}]
%
%
\begin{scope}[decoration={markings, mark=at position 0.5 with {\arrow{>}}}, evaluate={
	\l = 1;
	\cut=5;
},]
\draw [looseness=\l](2.,6.75) to [out=180, in=90] (1.25,4.75);
\draw [looseness=\l,white,line width=\cut](1.5,6) to [out=180, in=-90] (0.75,7.5);
\draw [looseness=\l](1.5,6) to [out=180, in=-90] (0.75,7.5);
\draw [looseness=\l](1.5,6) to [out=0, in=-90] (2.25,7.5);
\draw [looseness=\l,white,line width=\cut](2,6.75) to [out=0, in=90 ] (2.75,4.75);
\draw [looseness=\l](2.,6.75) to [out=0, in=90] (2.75,4.75);
\draw (2.75,4.75) -- (2.75,2.0);
\draw (1.25,4.75) -- (1.25,2.0);
\draw[] (2.75,2.0) to [out=-90, in=-90] (2.75-5.75,2.0);
\draw[] (1.25,2.0) to [out=-90, in=-90] (1.25-2.75,2.0);
\draw[postaction=decorate] (0.75,7.5) -- (0.75,8);
\draw (0.75,7.6) node[left] {\small $a$};
\draw (2.25,7.5) -- (2.25,8);
\draw[] (2.25,8) to [out=90, in=90] (2.25-5.25,8);
\draw[] (0.75,8) to [out=90, in=90] (0.75-2.25,8);
\draw (-3,2) -- (-3,8);
\draw (-1.5,2) -- (-1.5,8);
\end{scope}
\begin{scope}[decoration={markings, mark=at position 0.6 with {\arrow{>}}}, evaluate={
	\l = 1.2;
	\cut=5;
},]
\draw[white, line width=\cut] (2,4) ellipse (1.75cm and 0.5cm);
\draw[postaction={decorate}] (2,4) ellipse (1.75cm and 0.5cm);
\draw (0.5,3.6) node[below] {\small $x$};
\end{scope}
\begin{scope}[decoration={markings, mark=at position 0.0 with {\arrow{<}}}, evaluate={
	\l = 1.2;
	\cut=5;
},]
\draw[white, line width=\cut] (1.25,5) -- (1.25,4);
\draw[white, line width=\cut] (2.75,5) -- (2.75,4);
\draw[black, postaction=decorate] (1.25,5) -- (1.25,4);
\draw (1.25,5) node[left] {\small $a$};
\draw (2.75,5) -- (2.75,4);
\end{scope}
\end{tikzpicture}
=
\begin{tikzpicture}[baseline=0, thick,scale=.33, shift={(0,-4.8)}]
%
%
\begin{scope}[decoration={markings, mark=at position 0.5 with {\arrow{>}}}, evaluate={
	\l = 1;
	\cut=5;
},]
\draw [looseness=\l](2.,6.75) to [out=180, in=90] (1.25,4.75);
\draw [looseness=\l,white,line width=\cut](1.5,6) to [out=180, in=-90] (0.75,7.5);
\draw [looseness=\l](1.5,6) to [out=180, in=-90] (0.75,7.5);
\draw [looseness=\l](1.5,6) to [out=0, in=-90] (2.25,7.5);
\draw [looseness=\l,white,line width=\cut](2,6.75) to [out=0, in=90 ] (2.75,4.75);
\draw [looseness=\l](2.,6.75) to [out=0, in=90] (2.75,4.75);
\draw (2.75,4.75) -- (2.75,2.0);
\draw (1.25,4.75) -- (1.25,2.0);
\draw (2.75,2.0) to [out=-90, in=-90] (5,2.0);
\draw (1.25,2.0) to [out=-90, in=-90] (6.5,2.0);
\draw[postaction=decorate] (0.75,7.5) -- (0.75,8);
\draw (0.75,7.6) node[left] {\small $a$};
\draw (2.25,7.5) -- (2.25,8);
\draw (2.25,8) to [out=90, in=90] (5,8);
\draw (0.75,8) to [out=90, in=90] (6.5,8);
\draw (5,2) -- (5,8);
\draw (6.5,2) -- (6.5,8);
\end{scope}
\begin{scope}[decoration={markings, mark=at position 0.6 with {\arrow{>}}}, evaluate={
	\l = 1.2;
	\cut=5;
},]
\draw[white, line width=\cut] (2,4) ellipse (1.75cm and 0.5cm);
\draw[postaction={decorate}] (2,4) ellipse (1.75cm and 0.5cm);
\draw (0.5,3.6) node[below] {\small $x$};
\end{scope}
\begin{scope}[decoration={markings, mark=at position 0.0 with {\arrow{<}}}, evaluate={
	\l = 1.2;
	\cut=5;
},]
\draw[white, line width=\cut] (1.25,5) -- (1.25,4);
\draw[white, line width=\cut] (2.75,5) -- (2.75,4);
\draw[black, postaction=decorate] (1.25,5) -- (1.25,4);
\draw (1.25,5) node[left] {\small $a$};
\draw (2.75,5) -- (2.75,4);
\end{scope}
\end{tikzpicture}
=\widetilde{W}_{a\bar x}
	$$

(3):
\begin{align*}
\sum_{\mu} S^{(z)}_{(a\mu) (a\mu)} 
\,\,\,\,&=\sum_{\mu}
\frac{1}{D \sqrt {d_z}}
\begin{tikzpicture}[baseline=0, thick,scale=.33, shift={(0,-4.8)}]
%
\begin{scope}[decoration={markings, mark=at position 0.5 with {\arrow{>}}}, evaluate={
	\l = 1;
	\cut=5;
},]
\draw [looseness=\l](2.,6.75) to [out=180, in=90] (1.25,5.5);
\draw [looseness=\l,white,line width=\cut](1.5,6) to [out=180, in=-90] (0.75,7.5);
\draw [looseness=\l](1.5,6) to [out=180, in=-90] (0.75,7.5);
\draw [looseness=\l](1.5,6) to [out=0, in=-90] (2.25,7.5);
\draw [looseness=\l,white,line width=\cut](2,6.75) to [out=0, in=90 ] (2.75,5.5);
\draw [looseness=\l](2.,6.75) to [out=0, in=90] (2.75,5.5);
\end{scope}
\begin{scope}[decoration={markings, mark=at position 1 with {\arrow{>}}}, evaluate={
	\l = 1;
	\cut=5;
},]
\draw[] (2,4.5) to [out=70, in=-90] (2.75,5.5);
\draw[postaction=decorate] (2,4.5) to [out=110, in=-90] (1.25,5.5);
\draw (1.25,5.5) node[left] {\small $a$};
\end{scope}
\begin{scope}[decoration={markings, mark=at position 0.5 with {\arrow{>}}}, evaluate={
	\l = 1;
	\cut=5;
},]
\draw[postaction=decorate] (2,4.5)--(2,3.);
\draw (2,3.75) node[left] {\small $z$};
\draw (2,4.5) node[right] {\tiny $\mu$};
\draw (2,3) node[right] {\tiny $\mu$};
\draw[] (2,3) to [out=-70, in=90] (2.75,2.);
\draw[] (2.75,2.0) to [out=-90, in=-90] (5,2.0);
\end{scope}
\begin{scope}[decoration={markings, mark=at position 0.99 with {\arrow{>}}}, evaluate={
	\l = 1;
	\cut=5;
},]
\draw[] (2,3) to [out=-110, in=90] (1.25,2.);
\draw (1.25,2) node[left] {\small $a$};
\draw[black,postaction=decorate] (6.5,2.0) to [out=-90, in=-90] (1.25,2.0);
\end{scope}
\begin{scope}[decoration={markings, mark=at position 0.5 with {\arrow{>}}}, evaluate={
	\l = 1;
	\cut=5;
},]
\draw[] (0.75,7.5) -- (0.75,8);
\draw[] (2.25,7.5) -- (2.25,8);
\draw (2.25,8) to [out=90, in=90] (5,8);
\draw (0.75,8) to [out=90, in=90] (6.5,8);
\draw (5,2) -- (5,8);
\draw (6.5,2) -- (6.5,8);
\end{scope}
\end{tikzpicture}
\,\,\,\,=
\frac{1}{D}\frac{d_a}{d_z}
\begin{tikzpicture}[baseline=0, thick,scale=.33, shift={(0,-4.8)}]
%
%
\begin{scope}[decoration={markings, mark=at position 0.5 with {\arrow{>}}}, evaluate={
	\l = 1;
	\cut=5;
},]
\draw [looseness=\l](2.,6.75) to [out=180, in=90] (1.25,4.75);
\draw [looseness=\l,white,line width=\cut](1.5,6) to [out=180, in=-90] (0.75,7.5);
\draw [looseness=\l](1.5,6) to [out=180, in=-90] (0.75,7.5);
\draw [looseness=\l](1.5,6) to [out=0, in=-90] (2.25,7.5);
\draw [looseness=\l,white,line width=\cut](2,6.75) to [out=0, in=90 ] (2.75,4.75);
\draw [looseness=\l](2.,6.75) to [out=0, in=90] (2.75,4.75);
\draw (2.75,4.75) -- (2.75,2.0);
\draw (1.25,4.75) -- (1.25,2.0);
\draw (2.75,2.0) to [out=-90, in=-90] (5,2.0);
\draw (1.25,2.0) to [out=-90, in=-90] (6.5,2.0);
\draw[postaction=decorate] (0.75,7.5) -- (0.75,8);
\draw (0.75,8) node[left] {\small $a$};
\draw (2.25,7.5) -- (2.25,8);
\draw (2.25,8) to [out=90, in=90] (5,8);
\draw (0.75,8) to [out=90, in=90] (6.5,8);
\draw (5,2) -- (5,8);
\draw (6.5,2) -- (6.5,8);
\end{scope}
\begin{scope}[decoration={markings, mark=at position 0.56 with {\arrow{<}}}, evaluate={
	\l = 1.2;
	\cut=5;
},]
\draw[white, line width=\cut] (2,4) ellipse (1.75cm and 0.5cm);
\draw[postaction={decorate}] (2,4) ellipse (1.75cm and 0.5cm);
\draw (0.5,3.6) node[below] {\small $\omega_{\bar{z}}$};
\draw[white, line width=\cut] (1.25,5) -- (1.25,4);
\draw[white, line width=\cut] (2.75,5) -- (2.75,4);
\draw (1.25,5) -- (1.25,4);
\draw (2.75,5) -- (2.75,4);
\end{scope}
\end{tikzpicture}
\\
\,\,\,\, &=
\frac{1}{D}\frac{d_a}{d_z}\sum_x S_{0\bar z} S_{\bar zx}^*
\begin{tikzpicture}[baseline=0, thick,scale=.33, shift={(0,-4.8)}]
%
%
\begin{scope}[decoration={markings, mark=at position 0.5 with {\arrow{>}}}, evaluate={
	\l = 1;
	\cut=5;
},]
\draw [looseness=\l](2.,6.75) to [out=180, in=90] (1.25,4.75);
\draw [looseness=\l,white,line width=\cut](1.5,6) to [out=180, in=-90] (0.75,7.5);
\draw [looseness=\l](1.5,6) to [out=180, in=-90] (0.75,7.5);
\draw [looseness=\l](1.5,6) to [out=0, in=-90] (2.25,7.5);
\draw [looseness=\l,white,line width=\cut](2,6.75) to [out=0, in=90 ] (2.75,4.75);
\draw [looseness=\l](2.,6.75) to [out=0, in=90] (2.75,4.75);
\draw (2.75,4.75) -- (2.75,2.0);
\draw (1.25,4.75) -- (1.25,2.0);
\draw (2.75,2.0) to [out=-90, in=-90] (5,2.0);
\draw (1.25,2.0) to [out=-90, in=-90] (6.5,2.0);
\draw[postaction=decorate] (0.75,7.5) -- (0.75,8);
\draw (0.75,8) node[left] {\small $a$};
\draw (2.25,7.5) -- (2.25,8);
\draw (2.25,8) to [out=90, in=90] (5,8);
\draw (0.75,8) to [out=90, in=90] (6.5,8);
\draw (5,2) -- (5,8);
\draw (6.5,2) -- (6.5,8);
\end{scope}
\begin{scope}[decoration={markings, mark=at position 0.56 with {\arrow{<}}}, evaluate={
	\l = 1.2;
	\cut=5;
},]
\draw[white, line width=\cut] (2,4) ellipse (1.75cm and 0.5cm);
\draw[postaction={decorate}] (2,4) ellipse (1.75cm and 0.5cm);
\draw (0.5,3.6) node[below] {\small $x$};
\draw[white, line width=\cut] (1.25,5) -- (1.25,4);
\draw[white, line width=\cut] (2.75,5) -- (2.75,4);
\draw (1.25,5) -- (1.25,4);
\draw (2.75,5) -- (2.75,4);
\end{scope}
\end{tikzpicture}
\,\,\,\,=
\frac{d_a}{D^2}\sum_x  S_{zx} {\widetilde{W}}_{ax}
\end{align*}

(4): It follows from (3) by inverting the $S$-matrix.
\end{proof}

\subsection{Invariants determined by modular data}

\subsubsection{Sums of $R$ symbols and signs}

\begin{proposition}

\begin{enumerate}
    \item  For any modular category, the $R$ symbols satisfy: $$\sum_\mu [R^{aa}_c]_{\mu\mu}=\sum_{x,y,z} \frac{\theta_y^2}{\theta_a \theta_x^{2}} \frac{S_{0y} S_{az} S_{xz}^* S_{cx}^* S_{yz}^* }{S_{0z}} ,$$
and
$$\sum_\mu [R^{aa}_c]_{\mu\mu}= \frac{\theta_{c}^{1/2}}{\theta_{a}} \Lambda_{a,c},$$
where $\Lambda_{a,c} \in \mathbb{Z}$ is a signature and $\theta_{c}^{1/2}\Lambda_{a,c}$ is an invariant that is determined by the modular data. This generalizes the Frobenius-Schur indicator, which corresponds to $c=0$, i.e. $\kappa_{a}=\theta_{0}^{1/2}\Lambda_{a,0}$.

\item If the modular category is multiplicity-free, then,
$$R^{aa}_c=\sum_{x,y,z}\frac{\theta_y^2}{\theta_a \theta_x^{2}} \frac{S_{0y} S_{az} S_{xz}^* S_{cx}^* S_{yz}^* }{S_{0z}} = \pm \frac{\theta_{c}^{1/2}}{\theta_{a}},$$
where the sign $\Lambda_{a,c} = \pm 1$ of the square root in this expression is determined by the modular data.
\end{enumerate}
\end{proposition}

\begin{proof}
(1): On the one hand, we have


$$\btikz{thick,scale=.5,baseline=0}

\uotwostrandsigmaone{}{}{}{}{}{}

\begin{scope}[decoration={markings,mark=at position 0 with {\arrow{>}}}]
\draw[postaction={decorate}] (1,1.5) node[left] {$a$} to [out=90,in=90] (3.5,1.5);
\draw (3.5,1.5)--(3.5,0);
\begin{pgfonlayer}{top}
\draw[thick] (1,.05)--(1,-.5);
\draw[thick] (2,.05)--(2,-.5);
\draw[thick] (1,.-1)--(1,-1.5);
\draw[thick] (2,-1)--(2,-1.5);
\draw[thick] (2,-1.5) to [out=-90,in=-90] (3,-1.5)--(3,1.5) to[out=90,in=90] (2,1.5);
\draw[thick] (1,-1.5) to [out=-90,in=-90] (3.5,-1.5)--(3.5,0);
\end{pgfonlayer}
\end{scope}
\begin{scope}[decoration={markings,mark=at position .6 with {\arrow{<}}}]
\draw[thick,postaction={decorate}] (1.5,-.5) circle[x radius=1.25, y radius=.25];
\draw (0,-.75) node {$w_c$};
\end{scope}
\draw[white,fill=white] (.75,-.5) rectangle (1.25,0);
\draw[white,fill=white] (1.75,-.5) rectangle (2.25,0);

\etikz=\sum_{\mu} \sqrt{\frac{d_c}{d_a^2}} \btikz{thick,scale=.5,baseline=0}

\uotwostrandsigmaone{}{}{}{}{}{}

\begin{scope}[decoration={markings,mark=at position 0 with {\arrow{>}}}]
\draw[postaction={decorate}] (1,1.5) node[left] {$a$} to [out=90,in=90] (3.5,1.5);
\draw (3.5,1.5)--(3.5,0);
\begin{pgfonlayer}{top}
\draw[thick] (1,.05)--(1.5,-.5)--(2,.05);
\draw[thick] (1,-1.5)--(1.5,-1)--(2,-1.5);

\draw[thick] (2,-1.5) to [out=-45,in=-90] (3,-1.5)--(3,1.5) to[out=90,in=90] (2,1.5);
\draw[thick] (1,-1.5) to [out=225,in=-90] (3.5,-1.5)--(3.5,0);
\draw[thick] (1.5,-1)--(1.5,-.5);
\draw (1.5,-.75) node[left] {$c$};

\draw (1.5,-.5) node[right] {\tiny $\mu$};
\draw (1.5,-1) node[right] {\tiny $ \mu$};
\end{pgfonlayer}
\end{scope}

\draw[white,fill=white] (.75,-.5) rectangle (1.25,0);
\draw[white,fill=white] (1.75,-.5) rectangle (2.25,0);

\etikz=\sum_{\mu} [R^{aa}_c]_{\mu\mu} d_{c} .$$


On the other hand, we have

$$\btikz{thick,scale=.5,baseline=0}

\uotwostrandsigmaone{}{}{}{}{}{}

\begin{scope}[decoration={markings,mark=at position 0 with {\arrow{>}}}]
\draw[postaction={decorate}] (1,1.5) node[left] {$a$} to [out=90,in=90] (3.5,1.5);
\draw (3.5,1.5)--(3.5,0);
\begin{pgfonlayer}{top}
\draw[thick] (1,.05)--(1,-.5);
\draw[thick] (2,.05)--(2,-.5);
\draw[thick] (1,.-1)--(1,-1.5);
\draw[thick] (2,-1)--(2,-1.5);
\draw[thick] (2,-1.5) to [out=-90,in=-90] (3,-1.5)--(3,1.5) to[out=90,in=90] (2,1.5);
\draw[thick] (1,-1.5) to [out=-90,in=-90] (3.5,-1.5)--(3.5,0);
\end{pgfonlayer}
\end{scope}
\begin{scope}[decoration={markings,mark=at position .6 with {\arrow{<}}}]
\draw[thick,postaction={decorate}] (1.5,-.5) circle[x radius=1.25, y radius=.25];
\draw (0,-.75) node {$w_c$};
\end{scope}
\draw[white,fill=white] (.75,-.5) rectangle (1.25,0);
\draw[white,fill=white] (1.75,-.5) rectangle (2.25,0);

\etikz=\btikz{thick,scale=.5,baseline=10}

\draw[black] (.75,1.5) to [out=90, in=-90] (1.25,2);
\draw[black] (.75,-.5) to [out=90, in=-90] (1.25,0);
\draw[black] (1.25,0) to [out=90, in=-30] (1.125,.125);
\draw[black] (1.25,2) to [out=90, in=-30] (1.125,2.125);

\draw[black] (.875,2.25) to [out=-30, in=-90] (.75, 2.375);
\draw[black] (.875,.25) to [out=-30, in=-90] (.75, .375);

\draw[black] (.75,-.5) to [out=-90, in=-90] (.25,-.5);
\begin{scope}[decoration={markings,mark=at position .075 with {\arrow{>}}}]

\draw[postaction={decorate}] (.25,-.5)--(.25,2.375) to [out=90,in=90] (.75,2.375);
\end{scope}
\draw[black] (.75,.375)--(.75,1.5);
\uotwostrandsigmaone{black}{black}{black}{}{}{}
\draw (-.25,0) node {$w_c$};$$

\begin{scope}[decoration={markings,mark=at position .045 with {\arrow{>}}}]
\draw[black,postaction={decorate}] (1,2)--(1,2.5) node[right] {$a$} to [out=90,in=90] (4,2.5) --(4,-1) to [out=-90,in=-90] (1,-1);
\end{scope}
\draw[black] (1,-.5)--(1,-1);
\draw (2,1.5) to [out=90, in=90] (3,1.5) --(3,0) to [out=-90,in=-90] (2,0);

\etikz= \theta_a \btikz{thick,rotate=-90,scale=.33,baseline=10}
\begin{scope}[decoration={markings, mark=at position 0 with {\arrow{<}}}]
\draw[postaction={decorate}] (-5,0) node[left] {$a$}--(-1,0);
\end{scope}
\draw (3,0)--(1,0);
\begin{scope}
\draw (-5,0) to [out=180, in=180] (-5,2)--(3,2) to [out=0, in=0] (3,0);
\draw (-4,0) sin (-3,1);
\draw[white,line width=8]  (-4,0) to [out=237.5,in=180] (-4,-2)--(2,-2) to [out=0, in=-62.5] (2,0);
\draw (-4,0) to [out=237.5,in=180] (-4,-2);
\begin{scope}[decoration={markings, mark=at position 1 with {\arrow{<}}}]
\draw[postaction={decorate}] (-4,-2)--(2,-2) node[left] {$w_c$};
\draw (2,-2) to [out=0, in=-62.5] (2,0);
\end{scope}
\draw[white,line width=8]   (-1,-1) cos (-2,0);
\draw (-1,-1) cos (-2,0);
\draw (-1,-1) cos (0,0);
\draw (0,0) sin (1,1);
\draw[white,line width=8] (1,1) cos (2,0);
\draw (1,1) cos (2,0) ;
\draw[white,line width=8] (-1,0)--(1,0);
\draw (-1,0)--(1,0);
\draw[white,line width=8]  (-2,0) sin (-3,1);
\draw (-2,0) sin (-3,1);
\end{scope}
\draw[white,line width=8] (-5,0)--(-3,0);
\draw (-3,0)--(-5,0);
\etikz = 
\theta_a \sum_x S_{0c}S_{cx}^*
\btikz{thick,rotate=-90,scale=.33,baseline=10}
\begin{scope}[decoration={markings, mark=at position 0 with {\arrow{<}}}]
\draw[postaction={decorate}] (-5,0) node[left] {$a$}--(-1,0);
\end{scope}
\draw (3,0)--(1,0);
\begin{scope}
\draw (-5,0) to [out=180, in=180] (-5,2)--(3,2) to [out=0, in=0] (3,0);
\draw (-4,0) sin (-3,1);
\draw[white,line width=8]  (-4,0) to [out=237.5,in=180] (-4,-2)--(2,-2) to [out=0, in=-62.5] (2,0);
\draw (-4,0) to [out=237.5,in=180] (-4,-2);
\begin{scope}[decoration={markings, mark=at position 1 with {\arrow{<}}}]
\draw[postaction={decorate}] (-4,-2)--(2,-2) node[left] {$x$};
\draw (2,-2) to [out=0, in=-62.5] (2,0);
\end{scope}
\draw[white,line width=8]   (-1,-1) cos (-2,0);
\draw (-1,-1) cos (-2,0);
\draw (-1,-1) cos (0,0);
\draw (0,0) sin (1,1);
\draw[white,line width=8] (1,1) cos (2,0);
\draw (1,1) cos (2,0) ;
\draw[white,line width=8] (-1,0)--(1,0);
\draw (-1,0)--(1,0);
\draw[white,line width=8]  (-2,0) sin (-3,1);
\draw (-2,0) sin (-3,1);
\end{scope}
\draw[white,line width=8] (-5,0)--(-3,0);
\draw (-3,0)--(-5,0);
\etikz $$


\begin{align*}=\theta_a \sum_{x,y,\alpha} S_{0c}S_{cx}^{*} \sqrt{\frac{d_y}{d_x d_a}}
\hspace{5pt}
 \begin{tikzpicture}[thick,scale=.5,baseline=-40]
\braid[number of strands={2}] a_1^-1 a_1^-1 a_1^-1 a_1^-1;
\draw (2,-4.5) -- (1.5,-5);
\draw (1,-4.5) -- (1.5, -5);
\begin{scope}[decoration={markings, mark=at position 0 with {\arrow{>}}}]
\draw[postaction={decorate}] (2,0) node[right] {\small $a$} to [out=90, in =90] (3,0) -- (3,-5.5) to [out=-90, in=-45] (1.5,-5.5);
\end{scope}
\begin{scope}[decoration={markings, mark=at position 0.7675 with {\arrow{<}}}]
\draw[postaction={decorate}] (1,0) to [out=90, in =90] (0,0) -- (0,-5.5) node[left] {\small $x$} to [out=-90, in=-135] (1.5,-5.5);
\end{scope}
\begin{scope}[decoration={markings, mark=at position 0.75 with {\arrow{>}}}]
\draw[postaction={decorate}] (1.5,-5.5) node[right] {\tiny$\alpha$}--(1.5,-5)node[right]{\tiny $\alpha$};
\draw (1.5,-5.25) node[left] {\small $y$};
\end{scope}
\end{tikzpicture} &=\sum_{x,y,\alpha} S_{0c}S_{cx}^*\sqrt{\frac{d_y}{d_xd_a}} \frac{\theta_y^2}{\theta_x^2\theta_a} \hspace{5pt} \begin{tikzpicture}[baseline=0, thick,scale=.75]
\begin{scope}[decoration={markings, mark=at position 0.5 with {\arrow{>}}}]
\draw[postaction={decorate}] (0,-1) node[right] {\tiny$\alpha$}--(0,1)node[right]{\tiny $\alpha$};
\draw (0,0) node[left] {\small $y$};
\end{scope}
\begin{scope}[decoration={markings, mark=at position 0.1 with {\arrow{>}}}]
\draw[postaction={decorate}] (0,1)  to [out=45, in =90] (1.5,1)--(1.5,-1) to [out=-90,in=-45] (0,-1);
\end{scope}
\begin{scope}[decoration={markings, mark=at position 0.5 with {\arrow{<}}}]
\draw[postaction={decorate}] (0,1)  to [out=135, in =90] (-1.5,1)--(-1.5,-1) to [out=-90,in=225] (0,-1);
\draw (-1.5,0) node[left] {\small $x$};
\draw (0.1,1.55) node[right] {\small $a$};
\end{scope}
\end{tikzpicture} \\
&=d_c \sum_{x,y} N_{\bar{x}a}^y S_{0y}S_{cx}^* \frac{\theta_y^2}{\theta_x^2\theta_a}\end{align*}

It follows (using the Verlinde formula) that $$\sum_{\mu}[R_c^{aa}]_{\mu\mu} = \sum_{x,y,z} \frac{S_{xz}^* S_{az} S_{yz}^* }{S_{0z}} S_{0y} S_{cx}^* \frac{\theta_y^2}{ \theta_a \theta_x^{2} }.$$

The second property follows from the ribbon property $\left[R^{ab}_c R^{ba}_c\right]_{\mu \nu}=\theta_a^{-1}\theta_b^{-1}\theta_c \delta_{\mu \nu}$ (which holds for any pre-modular category), and is made evident by using the gauge freedom to diagonalize $R^{aa}_c$, in which case $\Lambda_{a,c}$ is the trace of the signature matrix $\frac{\theta_{a}}{\theta_{c}^{1/2}} \left[R^{aa}_c\right]$.

(2): Clear from (1).
\end{proof}

Similar results for unitary modular categories can be found in \cite{GR17}.

\subsubsection{Two-strand closures}

\begin{theorem}

\begin{enumerate}
    \item 
For knots that are two-braid closures colored by $a$, the colored knot invariant is: 
$$Z(\widehat{\sigma_1^{2n+1}})={\tilde{\textrm{Tr}}}_q[(R^{aa})^{2n+1}]=\sum_c d_c \left (\sum_\mu [R^{aa}_c]_{\mu\mu} \right ) {\left(\frac{\theta_c}{\theta_a^2}\right)}^n,$$ where $n$ is odd.

\item For links that are two-braid closures with the two components colored by $a,b$, the colored link invariant is:
$$Z(\widehat{\sigma_1^{2n}})={\tilde{\textrm{Tr}}}_q[(R^{ab}R^{ba})^n]=\sum_{c,z} d_c \left (\frac{S_{az}S_{bz}S_{cz}^*}{S_{0z}}\right ) {\left(\frac{\theta_c}{\theta_a\theta_b}\right)}^n.$$
\end{enumerate}
\end{theorem}

\begin{proof}
(1): 
$$\tilde{\textrm{T}}\text{r}_q[(R^{aa})^{2n+1}]=  \btikz{baseline=-25,scale=.5}
\begin{scope}[thick]
\braid a_1^{-1} a_1^{-1} a_1^{-1};
\draw (2,0) to [out=90,in=90] (3.5,0) --(3.5,-3.5) to [out=-90, in=-90] (2,-3.5);

\begin{scope}[decoration={markings, mark=at position 0 with {\arrow{>}}}]
\draw[postaction={decorate}] (1,0) node[left] {$a$} to [out=90,in=90] (4.5,0) --(4.5,-3.5) to [out=-90, in=-90] (1,-3.5);
\end{scope}
\end{scope}
\draw[dashed] (.5,-3.25) rectangle (2.5,-1.25);
\draw (2.5,-2) node[right] {$n$};
\etikz = \sum_{c,\mu} d_c  [R^{aa}_c]_{\mu\mu}\left ( \frac{\theta_c}{\theta_a^2} \right )^n = \sum_c d_c  \left (\sum_{\mu} [R^{aa}_c]_{\mu\mu}\right)\left ( \frac{\theta_c}{\theta_a^2} \right )^n$$
(2):\\
\begin{align*} 
\tilde{\text{T}}\text{r}_q[(R^{ab}R^{ba})^{n}] =\btikz{baseline=-20,scale=.5}
\draw[dashed] (.5,-2.25) rectangle (2.5,-.25);
\begin{scope}[thick]
\braid a_1^{-1} a_1^{-1};
\draw (1,0)--(1,.25);
\draw (2,0)--(2,.25);
\draw (1,-2.5)--(1,-2.75);
\draw (2,-2.5)--(2,-2.75);
\draw (2.5,-1.5) node[right] {$n$};
\begin{scope}[decoration={markings, mark=at position 0 with {\arrow{>}}}]
\draw[postaction={decorate}] (1,.25) node[left] { $a$} to [out=90,in=90] (4.5,.25) --(4.5,-2.75) to [out=-90, in=-90] (1,-2.75);
\draw[postaction={decorate}] (2,.25) node[right] {$b$} to [out=90,in=90] (3.5,0.25) --(3.5,-2.75) to [out=-90, in=-90] (2,-2.75);
\end{scope}
\end{scope}
\etikz = \sum_{c,\mu} d_c \left ( \frac{\theta_c}{\theta_a\theta_b} \right )^n&= \sum_c d_c N_{ab}^c   \left ( \frac{\theta_c}{\theta_a\theta_b} \right )^n  \\ 
 & = \sum_{c,z} d_c  \left (\frac{S_{az}S_{bz}S_{cz}^*}{S_{0z}} \right) \left ( \frac{\theta_c}{\theta_a\theta_b} \right )^n
\end{align*}
\end{proof}

\subsubsection{Frobenius-Schur indicators}

\begin{proposition}
The $n$th order Frobenius-Schur is determined by the modular data:

$$\nu_a^n=\frac{1}{D^2}\sum_{x,y}N_{ax}^yd_x d_y {\left(\frac{\theta_y}{\theta_x}\right)}^n.$$

\end{proposition}

\subsubsection{Invariants of lens spaces}

The invariant of any closed oriented $3$-manifold with a canonical $2$-framing is also an invariant of a modular category, in particular the invariants of the lens spaces $L(p,q)$, where $(p,q)=1$.  Those invariants are determined by the modular data as follows:

Expand $p/q$ as a continued fraction 

$$p/q = a_n - \cfrac{1}{a_{n-1} - \cfrac{1}{ \phantom{a_{n-2}} \ddots \phantom{\frac{1}{a_3}}
  \cfrac{1}{a_2 - \cfrac{1}{a_1}}}} .$$
Then $L(p,q)$ is obtained as Dehn surgery on the following framed $n$-component link with framings $\{a_i\}$:

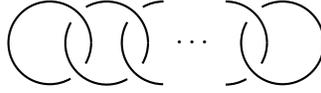
$$\btikz{baseline=0}

\begin{scope}[thick,scale=.55]
\draw (0,0) arc (0:-30:1);
\draw (0,0) arc (0:300:1);

\draw (1.375,0) arc (0:120:1);
\draw (1.375,0) arc (0:-30:1);
\draw (.875,-.875) arc (-60:-210:1);

\draw (1.75,1) arc (90:120:1);
\draw (.875,.5) arc (150:270:1);

\draw (2.5,0) node {\large$\cdots$};

\draw (4.25,0) arc (0:90:1);
\draw (4.25,0) arc (0:-30:1);
\draw (3.75,-.875) arc (-60:-90:1);

\draw (5.625,0) arc (0:120:1);
\draw (5.625,0) arc (0:-210:1);
\end{scope}

\etikz$$
Asigning the Kirby color $\omega_0$ to each component, we obtain:
$$Z(L_{(p,q)})=\frac{e^{-i \frac{\pi c}{4}\sigma(L)}}{D^{n+1}}\sum_{x_1,...,x_n}(\prod_{j=1}^n d_{x_j} (\theta_{x_j})^{a_j})\frac{S_{x_1x_2}}{d_{x_2}}\frac{S_{x_2x_3}}{d_{x_3}}\cdots \frac{S_{x_{n-2}x_{n-1}}}{d_{x_{n-1}}} S_{x_{n-1}x_n},$$ where $c$ is the chiral central charge, and $\sigma(L)$ the signature of the matrix:

$$
\begin{pmatrix}
a_1&-1&0&...&0&0&0\\
-1&a_2&-1&...&0&0&0\\
0&-1&a_3&-1&0&...&0\\
...&...&...&...&...&...&...\\
...&...&...&...&...&a_{n-1}&-1\\
0&0&...       &0&0   &-1&a_n
\end{pmatrix}
$$

\begin{proof}
The invariant $Z(L_{(p,q)})$, calculated by attaching the Kirby color $\omega_0$ onto each component \cite[Thm. 4.22]{ZW10}.  Expanding the formal sums leads to the above formula.
\end{proof}

\subsection{Combination of $F$ symbols}

All invariants of anyon models are polynomials equations of the $F$ and $R$ symbols.  However, concrete formulas can be complicated, as expressions of link invariants demonstrate. Also, it appears difficult to determine whether general intrinsic information, such as  $\sum_{\mu\nu}[F^{a,b,a}_{b}]_{b\mu\nu,b\nu\mu}$ or whether a certain (gauge-invariant) entry of an $F$-matrix is zero, is related to the triple $(S,T,W)$. Therefore, it might be challenging to discover a simple set of complete invariants of modular categories.

\section{Link invariants from Dijkgraaf-Witten TQFTs}

We first list some notations that are used in the sequel.

\begin{enumerate}

    \item $G(q,p;n)=\Z_q\rtimes_n \Z_p$, where $n\neq 1$ and $n^p=1 \mod q$, denotes the semi-direct product group of order $pq$. All such choices of $n$ give rise to isomorphic groups.
    \item $[{}^Gt]=\{gtg^{-1}\}$ is the conjugacy class of $t$ in $G$ (we sometimes denote this simply as $[t]$).
    \item $\mathcal{R}_G=\{t\}$ is a full representative set of conjugacy classes of $G$ so $G=\cup_{t\in \mathcal{R}_G}[{}^Gt]$.
    \item Let $C_G(t)=\{g|gt=tg\}$ be the centralizer of $t$ in $G$.  
    \item Let $R_t=\{r_{1,t},...,r_{n,t}\}=\{r_{i,t}\}$ with $r_{1,t}=e$ be a full representative set of $G/C_G(t)$ such that $[{}^Gt]=\{t_1,...,t_n\}$ with $t_1=t, t_i=r_{i,t}tr_{i,t}^{-1}$, where $e$ is the unit of $G$.
    \item $\omega\in Z^3(G;\U(1))$ is a normalized $3$-cocyle of a finite group $G$ with $\U(1)$ coefficient.
    \item $\vgo$ is the fusion category of $G$-graded vector spaces.
    \item $D^\omega(G)$ the Drinfeld double of a finite group $G$, which is a quasi-triangular quasi-Hopf algebra.  Physicists use the same notation for its representations category $\mcZ(G,\omega)$. In this paper, we sometimes use $\mathcal{D}^\omega(G)$ to denote its representation category.
    \item $\mcZ(\vgo)$ is the representation category of $D^\omega(G)$, which is modular and the same as the Drinfeld center of $\vgo$ and $\textrm{Rep}(\C^\omega [G])$.  The simple objects of $\mcZ(G,\omega)$ are denoted as $([t],\chi)$.
    \item $\mcZ(G,\omega)$, which is equivalent to $\mcZ(\vgo)$ as a modular category, is regarded as a concrete representation category of $D^\omega(G)$ whose irreducible representations (irreps) $V([t],\chi)$ corresponding to $([t],\chi)$ have bases $\{|r_{i,t}\ra \otimes |v_{a,t}\ra\}$ of dimension=$|[{}^Gt]|\textrm{dim}(V_\chi)$, where $|v_{a,t}\ra$ is a basis of the projecitve irrep $V_\chi$ labeled by $\chi$.  We will drop $\otimes$ if no confusion arises.
    
\end{enumerate}

In this section, we provide an exposition of the colored link invariants from the Dijkgraaf-Witten (DW) TQFTs and their computation from representations of colored braid groups based on quasi-triangular quasi-Hopf algebras $D^\omega(G)$ for finite groups\footnote{The DW theory with the trivial $3$-cocyle $\omega$ corresponds to the usual discrete gauge theory.  When the $3$-cocycle $\omega$ is non-trivial,  the corresponding theory is usually called twisted or deformed discrete gauge theory.  To save words, we will refer to all of them as discrete gauge theories.} $G$ \cite{DW90,DPR90}.  Such colored link invariants can be computed from two different representations of the braid groups: the anyonic representations using $F$ and $R$ symbols and fusion tree basis, and the representations from discrete gauge theories.  Those two braid group representations, referred to as the anyonic and discrete gauge theory ones, respectively, are not the same, but equivalent to each other in the sense of quasi-localization  \cite{RW12,GHW12}.  The discrete gauge theory representation is explicitly local, in that each braid strand has its own attached Hilbert space, while the anyonic state space representation is generally not a tensor product.  The discrete gauge theory representation is a quasi-localization of the anyonic representation such that the constituent irreps of both braid group representations are the same as a set without multiplicities.  To compute the colored link invariants from the two representations, we use the Markov trace for the anyonic representations, but simply the linear algebra trace for the discrete gauge theory ones.

\subsection{Two braid group representations from doubles of finite groups}

The topological representations are the anyonic ones, and the discrete gauge theory ones are the local representations from the quasi-triangular quasi-Hopf algebra $D^\omega(G)$. The new invariants for modular categories that we propose can be computed as link invariants from the associated DW TQFTs.

To avoid confusion, we will denote the same modular category from doubles $D^\omega(G)$ of finite groups $G$ and their simple objects differently. When regarded as an anyon model, we will denote the modular category of the representations of $D^\omega(G)$ as $\mcZ(\vgo)$, and each irrep of $D^\omega(G)$ is an anyon.  Since the irreps are used as labels of anyons, they are denoted as pairs $([t], \chi)$, where $[t]$ represents a conjugacy class $[{}^Gt]$, and $\chi$ is some projective irrep of the centralizer $C_G(t)$ of $t$.  We consider $\chi$ as a projective character of the relevant group, and use it interchangeably as the irrep that it determines:  $\pi_\chi$ the irrep as a group homomorphism or $V_\chi$ the irrep as a module or sometimes even $V_{\pi_\chi}$.   In the discrete gauge theory setting, a basis of the irreps $([t],\chi)$ is often used to represent, in physical jargon, the dyonic (i.e. flux and charge) degrees of freedom, so we will denote the irrep as $V([t],\chi)$ with a basis $|r_{i,t}\rangle |v_{a,t,\chi}\rangle$, where $|r_{i,t}\rangle$ with $r_{i,t}\in R_t$ is a basis of $\C[R_t]$ and $|v_{a,t,\chi}\rangle$ is a basis of $V_\chi$. 

The anyonic representations of the braid groups have the same asymptotic dimensions as the discrete gauge theory ones, but are smaller in general.  Consider, for example, the $D(S_3)$ theory for the two strand braid groups $B_2$ labeled by the anyon $C$ of quantum dimension=$2$. Since $C\otimes C=A\oplus B \oplus C$, the anyonic representation of $B_2$ is $3$-dimensional corresponding to the three fusion channels $A,B,C$, while the discrete gauge theory one is $4$-dimensional because the irrep $C$ is $2$-dimensional. The relation between the anyonic and dyonic pictures is understood. The former is recovered from the later by imposing the corresponding gauge invariance, i.e. restricting to the subset of state space and operators that respect the gauge ``symmetry." For example, pair creation of $C$ and $C$ from vacuum does not give two independent $2$-dim irreps, but rather the state $\sum_{v \in C} |v\ra |v\ra$, and there is no way to physically ``look inside" the irrep $C$ carried locally by either anyon, i.e there are no such topological states as $|v\ra$ by itself.

The anyonic representations of the colored braid groups describe the statistics of anyons in the anyon model, whose computation requires the full set of the $F$ and $R$ symbols.  There are very few modular categories for the DW theories for which the full set of $F$ and $R$ symbols are known.  Besides the pointed ones, the only non-Abelian anyon models that authors know their full set of $F$ and $R$ symbols are the $D(S_3)$ theories \cite{MTthesis}.  If we can find all the $F$ and $R$ symbols of the Mignard-Schauenburg examples, then the punctured $S$-matrices can be calculated and shown that they are invariants beyond the modular data.  Since we do not know the full set of $F$ and $R$ symbols for the Mignard-Schauenburg examples, in the following we will not discuss their anyonic representations of the braid groups unless stated explicitly.

\subsubsection{Operator invariants of ribbon graphs and link invariants}

A far-reaching generalization of the Jones representation of the braid groups and the Jones polynomial of knots is via ribbon category theory \cite{TuBook}.  Given any ribbon category $\cB$, there is an associated tensor functor $\mathcal{T}_{\cB}$ from the category of $\cB$-colored framed tangles to $\cB$.
If the tangle $L$ is a framed closed link, then (as a morphism) $\mathcal{T}_\cB(L)\in \textrm{Hom}(\uno,\uno)$, where $\uno$ is the tensor unit of $\cB$. Hence, $\mathcal{T}_\cB(L)=\lambda \cdot \textrm{id}_\uno$ and the complex number $\lambda$ is the link invariant.
In \cite{AC92}, the authors describe such an operator invariant from any quasi-triangular quasi-Hopf algebra $\mathcal{A}$.  We refer the readers to \cite{AC92} for the general theory.  In the following, an outline of the computation of such operator invariants for colored braids is presented following \cite{AC92}.

A coloring of a braid $b$ is an assignment of an irrep $V_i$ of $\mathcal{A}$ to each strand of $b$.  Since the representation category of $\mathcal{A}$ is, in general, not strict and not strictified, an order of the braid strand is needed to parenthesize the tensor products of $\{V_i\}$.  By default, we will always parenthesize all left $((V_1\otimes V_2)\otimes V_3)...)))$.
The operator invariant of the braid $b$ will be an intertwiner operator of $((V_1\otimes V_2)\otimes V_3)...)))$.  The braid $\sigma_i$ is implemented on the tensor product with the default all left parenthesis and let $c_{V_i,V_{i+1}}$ denote the operator associated with the braid generator $\sigma_i$.  In practice, before applying $c_{V_i,V_{i+1}}$, the associators are used to prepare the parenthesis into the default order first.  Let $\phi=\sum \phi^1\otimes \phi^2 \otimes \phi^3 \in \mathcal{A}\otimes \mathcal{A}\otimes \mathcal{A}$ be the associator and $\Delta$ the comultiplication.  Define $\psi_i=\Delta^{i-2}(\phi)\otimes {\textrm{id}}^{\otimes (n-i-1)}$ on $((V_1\otimes V_2)\otimes V_3)...V_n)))$, then $\psi_i$ determines a parenthesis that the $i$th and $(i+1)$th spaces are parenthesized together and the first $(i-1)$ spaces are parenthesized all left for $\Delta^{(i-2)}$ so that $\Delta$ acts on $\phi^1$ consecutively if $i\geq 3$.  The parenthesis is extended to the remaining tensor factors by parenthesis all left first after the left $(i+1)$ tensor factors, but it does not matter as $\psi$ is the identity operator afterwards.  Then $$c_{V_i,V_{i+1}}=\psi_i^{-1}\cdot P\cdot R\cdot \psi_i$$ if $i\geq 3$, where $P$ is the flipping of two tensor factors. For $i=1$ and $2$, $c_{V_1,V_{2}}=P\cdot R$, and $c_{V_2,V_{3}}=\phi^{-1} \cdot P\cdot R \cdot \phi$.

\subsubsection{Operator invariants of braids and links from discrete gauge theory}

In this section, the detail for the computation of the operator invariants of colored braids from the doubles of finite groups is provided in order to obtain the $W$-matrix.

Let $G$ be a finite group with a normalized $3$-cocycle $\omega \in Z^3(G,\U(1))$, i.e. $\omega(g,h,k)=1$ if any of $g,h,k$ is the group unit $e\in G$. The twisted double $D^\omega(G)=F(G)\rtimes \C[G]$ is a quasi-triangular quasi-Hopf algebra \cite{DPR90}.  

As a vector space, $D^\omega(G)$ is simply $F(G)\otimes \C[G]$ of dimensions $|G|^2$, where $F(G)$ is the abelian algebra of functions of $G$ and $\C[G]$ the group algebra.  The constant function on $G$ with value $=1$ is simply denoted as $1$, which is the unit of $F(G)$ and equal to $\sum_{g\in G}\delta_g$, where $\delta_g(h)=\delta_{g,h}$ is the delta function on $G$. The basis elements of the vector space $F(G)\otimes \C[G]$ will be denoted as $P_xy$, which is really $\delta_x\otimes y$, for $x,y\in G$.  

The action of $G$ on $F(G)$ is encoded in the multiplication $$P_gxP_hy=\delta_{g,xhx^{-1}}\theta_{g}(x,y)P_g (xy)$$ by an induced $2$-cochain:
\begin{equation}
  \theta_g(x,y)=\frac{\omega(g,x,y)\omega(x,y,(xy)^{-1}gxy)}{\omega(x,x^{-1}gx,y)}.  
\end{equation}

The comultiplication of $D^\omega(G)$ is given as
$$\Delta(P_gh)=\sum_{xy=g}\gamma_{h}(x,y) P_xh\otimes P_yh,$$ 
where the $2$-cochain $\gamma_h(x,y)$ is 
\begin{equation}
    \gamma_{h}(x,y)=\frac{\omega(x,y,h)\omega(h,h^{-1}xh,h^{-1}yh)}{\omega(x,h,h^{-1}yh)}.
\end{equation}

The unit of $D^\omega(G)$ is $P_1e=1\otimes e=\sum_{g\in G}P_ge$, and the associativity for left default parenthesis is given by the invertible element $\phi \in (D^\omega(G))^{\otimes 3}$ such that  
\begin{equation}
   \phi=\sum_{g,h,k}\omega(g,h,k)^{-1}P_ge\otimes P_he\otimes P_ke. 
\end{equation}

$D^\omega(G)$ has a universal $R$-matrix $R$ given by 
\begin{equation}
   R=\sum_{g,h}P_ge\otimes P_hg,
\end{equation}
and its inverse \cite{ERW08} is
\begin{equation}
  R^{-1}=\sum_{g,h}\theta_{ghg^{-1}}(g,g^{-1})^{-1}P_ge\otimes P_hg^{-1}.  
\end{equation}

It is well-known that irreps of $D^\omega(G)$ are labeled by pairs $([t],V_{\pi_t})$, where $V_{\pi_t}$ is a $\theta_t$-projective irrep of $C_G(t)$ (which is also denoted as $V_{\chi_t}$ sometimes).
An explicit description of the operator invariants of colored braids requires specific bases of those irreps.  Suppose $V_{\pi_t}$ is a $\theta_t$-projective irrep of $C_G(t)$ so that 

\begin{equation}\label{piequ}
     \pi_t(xy)=\theta_t(x,y)^{-1}\pi_t(x)\pi_t(y),
\end{equation}
 and $|v_{a,t}\ra$ is a basis of $V_{\pi_t}$.  The irrep of $D^\omega(G)$ labeled by the pair $([t],V_{\pi_t})$ is induced from the $\theta_t$-projective irrep $V_{\pi_t}$ of $C_G(t)$, called the DPR induction \cite{DPR90}.  The induction to an irrep $\C[R_t]\otimes V_{\pi_t}$ of $D^\omega(G)$ uses $|r_{i,t}\ra |v_{a,t} \ra$ as an explicit basis of the  
induced irrep of $D^\omega(G)$.  The amazing formulas for the DPR induction can be found in \cite{ACM04}:

\begin{proposition}\label{prop:Induction}
The DPR induction of a $\theta_t$-projective irrep $\pi_t$ of $C_G(t)$ to $D^\omega(G)$ is given by the following action of $P_xy$ on the basis $|r_{i,t}\ra |v_{a,t} \ra$:

\begin{equation}
    P_xy(|r_{i,t}\ra |v_{a,t} \ra)=\delta_{x,r_{j,t}tr_{j,t}^{-1}}\theta_x(y,r_{i,t})\theta_x(r_{j,t},s_t)^{-1}(|r_{j,t}\ra \cdot \pi_t(s_t)|v_{a,t} \ra),
\end{equation}

where $s_t\in C_G(t)$ such that $yr_{i,t}=r_{j,t}s_t$.

\end{proposition}

The computation of the $W$-matrix needs only the operator invariants for the $2$ and $3$ strand braids, therefore, only these cases will be provided in full detail.

The associativity $\phi$ acts as 

\begin{equation}
  (|r_{i,x}v_{a,x}\ra|r_{j,y}v_{b,y}\ra) |r_{k,z}v_{c,z}\ra=\omega(x_i,y_j,z_k)^{-1} |r_{i,x}v_{a,x}\ra (|r_{j,y}v_{b,y}\ra |r_{k,z}v_{c,z}\ra).  
\end{equation}

The action of the braid $\sigma_i, i=1,2$, on two irreps $([x],V_{\pi_x})\otimes ([y],V_{\pi_y}), x,y\in G$ with basis $|r_{i,x}\ra |v_{a,x}\ra \otimes |r_{k,y}\ra |v_{b,y}\ra$ is through $P\cdot R$, where $P$ is the flip of two tensor factors.

Note that the universal $R$-matrix acts non-trivially only when $g_R=r_{i,x}xr_{i,x}^{-1}=x_i, h_R=r_{l,y}yr_{l,y}^{-1}=y_l$, where $l$ and $s_{y_{kl}}$ are uniquely determined by $x_ir_{k,y}=r_{l,y}s_{y_{kl}}$.  Therefore, 

\begin{equation}
    c_{x,y}: |r_{i,x}\ra |v_{a,x}\ra \otimes |r_{k,y}\ra |v_{b,y}\ra \rightarrow \Omega_+(x,y,i,k)|r_{l,y}\ra|v_{b,y}\ra \otimes |r_{i,x}\ra|v_{a,x}\ra,
\end{equation}
where 
\begin{equation}
\Omega_+(x,y,i,k)=\theta_h(g,r_{k,y})\theta_h(r_{l,y},s_{y_{kl}})^{-1}\pi_y(s_{y_{kl}})    
\end{equation} 
if $V_\pi$ are all $1$-dimensional.  Otherwise, $\pi_y(s_{y_{kl}})$ applies to the basis $|v_{b,y}\ra$ resulting in a linear combination.

To implement the inverse braiding, we apply $R^{-1}\cdot P^{-1}$---the inverse of $P\cdot R$---to $([y],V_{\pi_y})\otimes ([x],V_{\pi_x}), y,x \in G$ with basis $|r_{l,y}\ra|v_{b,y}\ra \otimes |r_{i,x}\ra |v_{a,x}\ra,$ so the inverse universal $R$-matrix $R^{-1}$ is applied to 
$$|r_{i,x}\ra |v_{a,x}\ra \otimes |r_{l,y}\ra |v_{b,y}\ra,$$
It follows that

\begin{equation}
    c_{x,y}^{-1}:|r_{l,y}\ra |v_{b,y}\ra \otimes |r_{i,x}\ra |v_{a,x}\ra \rightarrow \Omega_-(y,x,l,i)|r_{i,x}\ra |v_{a,x}\ra \otimes |r_{k,y}\ra |v_{b,y}\ra,
\end{equation}
where $k,s_{y_{lk}}$ are uniquely determined by 
$x_i^{-1}r_{l,y}=r_{k,y}s_{y_{lk}}$, and

\begin{equation}
    \Omega_-(y,x,l,i)=\theta_{ghg^{-1}}(g,g^{-1})^{-1} \theta_h(g^{-1},r_{l,y})\theta_h(r_{k,y},s_{y_{lk}})^{-1}\pi_y(s_{y_{lk}}).
\end{equation}
In this case, 
$g_{R^{-1}}=r_{i,x}xr_{i,x}^{-1}=x_i,h_{R^{-1}}=r_{k,y}yr_{k,y}^{-1}=y_k.$

\section{Towards a complete invariant of modular categories}

In this section, we prove the following:
\begin{theorem}
Neither the $W$-matrix nor the set of all punctured $S$-matrices of a modular category is determined by its modular data $(S,T)$ in general. 
\end{theorem}

The theorem follows from the observation that the $W$-matrix distinguishes some counterexamples in \cite{Mignard-Shauenburg} which have the same modular data. This provides a different proof that modular data do not uniquely determine a modular category.  More specifically, the theorem follows from Sect.~\ref{nomatch} and Prop.~\ref{Wmatrixproperties}.

\subsection{Mignard-Schauenburg modular categories}\label{Main example}

The Mignard-Schauenburg modular categories (MS-MCs) $\mcZ(\vgou)$ for $G=D^\omega(\Z_q\rtimes_n \Z_p)$ are modular categories that go beyond modular data.  

The MS-MCs are the representation categories of doubles $D^\omega(\Z_q\rtimes_n \Z_p)$ of $G(q,p;n)=\Z_q\rtimes_n \Z_p$ and $\omega \in Z^3(G(q,p;n);\U(1))$, where $p,q$ are odd primes such that $p|(q-1)$ and $n$ is such that $n^p=1 \mod q$ and $n\neq 1$.
The generator of $\Z_p$ acts on the generator of $\Z_q$ by multiplication of the element $n\in \mathbb Z_q$ of multiplicative order $p$. Elements of $G$ are of the form $(a^l,b^m)\equiv a^lb^m$ and multiplication is given by
$$(a^l,b^m)\cdot (a^{l'},b^{m'})=(a^{l} (b^m a^{l'}b^{-m}),b^{m}b^{m'})=(a^{l+n^m l'},b^{m+m'}).$$

There are $p$ distinct modular categories $\mcZ(\vgou)=\textrm{Rep}(D^{\omega^u}(G))$, labeled by $u=0,\cdots ,p-1$, where $\omega^u$ are 3-cocycles representing the $p$ different cohomology classes of $H^3(G(q,p;n),\U(1))\cong \Z_p$. On the other hand, there are only $3$ distinct sets of different modular data (when allowing relabeling of simple objects).  In this section, we first compute the basic invariants of those modular categories, and then the simplest example $G(11,5;4)$ is studied.

The Drinfeld center $\mcZ(\textrm{Vec}_G^{\omega^u})$ has simple objects labeled by pairs $([t],\chi_t)$ where $t$ is a representative of the conjugacy class $[t]$ and $\chi_t$ is a irreducible $\omega^u_g$-projective representation of the centralizer $C_G(t)$. 

For the case $G(11,5;4)$, there are $\frac{q^2-1+p^3}{p}=49$ simple object types. The seven conjugacy classes, their centralizers and their sizes are
\begin{equation}
	\begin{array}{c|c|c}
	\textrm{Conj. Class} & \textrm{Centralizer} & |\textrm{Cong. Class}|\cdot |\textrm{Centralizer}|\\ \hline
	\left[1\right] & G & 1 \cdot 55\\
	\left[ a^1 \right] & \Z_{11} & 5 \cdot 11\\
	\left[ a^2 \right] & \Z_{11}& 5 \cdot 11\\
	\left[ b^1 \right] & \Z_{5}& 11\cdot 5\\
	\left[ b^2 \right] & \Z_{5}& 11 \cdot 5\\
	\left[ b^3 \right] & \Z_{5}& 11 \cdot 5\\
	\left[ b^4 \right] & \Z_{5}& 11 \cdot 5
	\end{array}
\end{equation}
E.g. $[a^1]=\{a,a^3,a^4,a^5,a^9\}$, $[a^2]=\{a^2,a^6,a^7,a^8,a^10\}$, $[b^k]=\{a^lb^k\}$ for $l=0,\ldots,10$.

We label the simple objects or anyons as $I_{0,...,6}\equiv([1],\chi_j)$, $A_{l,m}\equiv([a^l],\omega_{11}^m)$ and $B_{k,n}\equiv([b^k],\omega_5^n)$.  Besides $\uno$, none of the $49$ simple objects is self-dual, so they form $24$ pairs of dual anyons.

\subsubsection{Quantum dimensions}
Their quantum dimensions are
\begin{equation}
\begin{array}{c|c|c|c}
\textrm{Conj. Class} & \textrm{Irrep} & d=|\textrm{Cl}| \cdot \dim(\textrm{Irrep}) & \textrm{\# of this type} \\ \hline
\left[1\right] & \chi_0,\ldots,\chi_4 & 1\cdot 1=1 & 1\cdot 5=5\\ 
\left[1\right] & \chi_5,\chi_6 & 1\cdot 5=5 & 1\cdot 2=2\\
\left[ a^{1} \right],\left[ a^{2} \right] & \omega_{11}^0,\ldots,\omega_{11}^{10} & 5\cdot 1 = 5 & 2\cdot 11=22\\
\left[ b^{1} \right],\left[ b^{2} \right],\left[ b^{3} \right],\left[ b^{4} \right]  & \omega_5^0,\ldots,\omega_5^4 & 11\cdot 1=11 & 4\cdot 5=20
\end{array}.
\end{equation}
There are a total of 5 abelian anyons and 44 non-abelian ones. The total quantum dimension is 
\begin{equation}
	{D} = |G|=55
\end{equation}

\subsubsection{Topological twists}\label{twists}
The twists can be calculated from Eq. (5.21) of \cite{Coste:2000tq}
\begin{equation}
	\theta_{(g,\chi)} = \frac{\textrm{tr}\chi(g)}{\textrm{tr}\chi(e)} \epsilon_g(g)
\end{equation}
($\epsilon_g(h)$ is the $\mu_{b^k}(x)=\exp(\frac{2\pi i}{p^2} ku [x]_p)$ of \cite{Mignard-Shauenburg}). 

\begin{equation}
\begin{array}{c|c|c|c}
\textrm{Conj. Class} & \textrm{Irrep} & d_{(g,\chi)}& \theta_{(g,\chi)}\\ \hline
\left[1\right] & \chi_{0},...,\chi_4 & 1&1\\ 
\left[1\right] & \chi_5,\chi_6 & 5&1\\
\left[ a^{l=1,2} \right]& \omega_{11}^m &  5&\exp[\frac{2\pi i}{q}lm ]\\
\left[ b^{k=1,2,3,4} \right] & \omega_5^n & 11&\exp[\frac{2\pi i}{p}nk] \exp[\frac{2\pi i}{p^2}k^2u ]
\end{array}.
\end{equation}
The chiral central charge is $\frac{1}{\mathcal{D}} \sum_{a\in D(G)} d_a^2\theta_a =1\implies c_-=(0 \mod 8)$.

Explicitly, the $T$ matrices are given in Table (\ref{explicitT}) in the appendix.

\subsection{Formulas for the $\theta_t(x,y)$-projective characters $\pi_t$}
The operator invariants of colored braids are determined by the $\theta_t(x,y)$-projective representation $\pi_t$ of centralizers.  In our case of $G(11,5;4)$, the irreps of the centralizers are all $1$-dimensional except for the $I$ anyons, which are not important for our discussion.  Therefore, the irreps $\pi_t$'s are simply projective characters.  In this section, we provide formulas for those projective characters.

The $3$-cocyles for the MS-MCs are

\begin{align}
 \omega_u(g,h,k)&=  e^{\frac{2\pi i}{p^2} u[k_b]_p([g_b]_p+[h_b]_p - [g_b+h_b]_p)}\\ &=
\begin{cases}
1, \quad &\text{ if } [g_b]_p+[h_b]_p < p,\\
e^{\frac{2\pi i}{p}u[k_b]_p}, \quad &\text{ if } [g_b]_p+[h_b]_p\geq  p,
\end{cases}       \notag
\end{align}
where $g=a^{g_a} b^{g_b}$ and $[x]_p = x\mod p$.
The centralizer of $b^k$ in $G(q,p;n)$ is $\langle b  \rangle$ and the associated 2-cocycle (in fact a 2-coboundary) is

\begin{align*}
 \theta^u_{b^k}(b^{i_1},b^{i_2})&=\omega_u(b^{i_1},b^{i_2},b^k)\\ &=\begin{cases}
1, \quad & \text{ if } i_1+i_2< p\\
e^{\frac{2\pi i}{p}uk}, \quad & \text{ if } i_1+i_2\geq p,
\end{cases} 
\end{align*}where $i_1, i_2 \in \{0,\ldots p-1\}$.

Given $s\in \{0,1\ldots p-1\}$ define 
\begin{align}\label{projectivesolution}
\pi_{k}^{s}(b^l)=e^{\frac{2 \pi i}{p}sl}e^{\frac{2 \pi i}{p^2}ukl}=e^{\frac{2\pi i}{p^2}(sp+uk )l}, && 0\leq l<p
\end{align} The functions $\pi_{k}^{s}$ satisfy \eqref{piequ}. Thus each pair $(b_k,\pi_{k}^s)$ defines a simple object $B_{k,s}$, that explicitly can be realized over the vector space $V(k,s)=\C[\Z_q]$ and has associated braiding
\begin{equation}\label{formula:braiding}
    c_{V(k,s)}: | x\rangle \otimes |y \rangle \mapsto \pi_k^s(b^k)| (1-n^k)x+n^ky\rangle \otimes |x \rangle,
\end{equation}
and associativity constraint 
\begin{equation}\label{formula:associativity}
\phi: | x\rangle| y\rangle| z\rangle \mapsto \omega_u(b^k,b^k,b^k)^{-1} | x\rangle| y\rangle| z\rangle
\end{equation}

Computing the operator invariants of colored braids needs the evaluation of 
$\pi_{y}(s_y)$ and will be done as follows. First the centralizer $C_G(y)$ of $y$ is calculated, and then its character table is constructed. The linear character is the entry corresponding to the $n$'th irrep evaluated on the conjugacy class associated to $s_y$. To get the projective character associated to a $\theta_t$ factor set Eq. (\ref{projectivesolution}), an additional factor is multiplied onto the linear character.

For the example at hand, there are three possible centralizers, $C_G(e) = G$, $C_G(a^lb^0) = \mathbb Z_{11}$, and $C_G(a^l b^k) = \mathbb Z_{5}$, where $l=0,...,10$ and $k=1,...,4$. Since $\theta_t(x,y)$ is non-trivial only on group elements of the form $a^lb^k$, only the characters from the $\mathbb Z_5$ centralizer are projective.

Let $\xi_m$ be the principal $m$'th root of unity, i.e. $\xi_m= e^{\frac{2\pi i}{m}}$.

The $G$ character table is
\begin{equation}
\begin{array}{c|ccccccc}
C_G(e') = G & [e] & [a^1] & [a^2] & [b^1] & [b^2] & [b^3] & [b^4] 
\\\hline
\chi_0&1 & 1 & 1 & 1 & 1 & 1 & 1 \\
\chi_1&1 & 1 & 1 & \xi_5  & \xi_5 ^2 & \xi_5 ^3 & \xi_5 ^4  \\
\chi_2&1 & 1 & 1 & \xi_5 ^2 & \xi_5 ^4 & \xi_5  & \xi_5 ^3  \\
\chi_3&1 & 1 & 1 & \xi_5 ^3 & \xi_5  & \xi_5 ^4 & \xi_5 ^2 \\
\chi_4&1 & 1 & 1 & \xi_5 ^4 & \xi_5 ^3 & \xi_5 ^2 & \xi_5 \\
\chi_5&5 & \sigma & \sigma^* & 0 & 0 & 0 & 0  \\
\chi_6&5 & \sigma^* & \sigma & 0 & 0 & 0 & 0  
\end{array}
\end{equation}
where $\sigma=e^{\frac{2\pi i}{11} 1}+e^{\frac{2\pi i}{11} 3}+e^{\frac{2\pi i}{11} 4}+e^{\frac{2\pi i}{11} 5}+e^{\frac{2\pi i}{11} 9}$.

The $\mathbb Z_{11}$ character for the $\chi_m$-$[a^l]$ entry is equal to $\xi_{11}^{lm}= e^{\frac{2\pi i}{11} lm}$, where $l,m=0,1,\ldots,10$.

The linear $\mathbb Z_{5}$ character for $\chi_n$-$[b^k]$ is equal to $\xi_5^{nk} = e^{\frac{2\pi i}{5} nk}$.
For this to be $\theta_t(x,y)^u$ projective, $\pi_{t,n}(xy) = \theta_{t}(x,y)^{-u} \pi_{t,n}(x) \pi_{t,n}(y)$, multiply each character by
\begin{equation}
    \mu_t^u(g) = \mu_{a^lb^k}^u(a^{l'}b^{k'}) = e^{\frac{2\pi i}{25} k k'u}
\end{equation}
Then, the $\mathbb Z_5$ \emph{projective} character table is 
\begin{equation}
\begin{array}{c|ccccc}
C_G(t=a^{l'}b^{k'})=\mathbb{Z}_5 & [e] & [b^1] & [b^2] & [b^3] & [b^4] \\ \hline
\chi_0& 1 & 1 & 1 & 1 & 1 \\
\chi_1& 1 & \xi _5 \cdot \mu_t^u(b^1) & \xi _5^2 \cdot \mu_t^u(b^2)& \xi _5^3 \cdot \mu_t^u(b^3)& \xi _5^4 \cdot \mu_t^u(b^4)\\
\chi_2& 1 & \xi _5^2 \cdot \mu_t^u(b^1)& \xi _5^4 \cdot \mu_t^u(b^2)& \xi _5 \cdot \mu_t^u(b^3)& \xi _5^3 \cdot \mu_t^u(b^4)\\
\chi_3& 1 & \xi _5^3 \cdot \mu_t^u(b^1)& \xi _5 \cdot \mu_t^u(b^2)& \xi _5^4 \cdot \mu_t^u(b^3)& \xi _5^2 \cdot \mu_t^u(b^4)\\
\chi_4& 1 & \xi _5^4 \cdot \mu_t^u(b^1)& \xi _5^3 \cdot \mu_t^u(b^2)& \xi _5^2 \cdot \mu_t^u(b^3)& \xi _5 \cdot \mu_t^u(b^4)\\
\end{array}
\end{equation}
The $\chi_n$-$[b^k]$ entry for the $C_G(t=a^lb^{k'})$ table is $\xi_5^{nk} \mu_t^u(b^k) = e^{\frac{2\pi i}{25} (5nk + kk' u)} $. When $u=0$, this reduces to the linear character that was first given.

It follows that the $\pi_t$'s that satisfy Eq. (\ref{piequ}) is

\begin{equation}
    \pi_{t,n,\omega^u}(x)=e^{\frac{2\pi i}{25} (5nk + kk' u)}
\end{equation}
for  $t=a^{l'}b^{k'},x=a^lb^k,$ and
$\theta_{t,u}(x,y)=\omega^u(x,y,t)$ if $C_G(t)=\Z_5$.  Otherwise $\pi_{t,n,\omega^u}$ is the linear character and $\theta_{t,u}(x,y)=1$.

\subsection{Link invariants from MS-MCs $\mcZ(\vgou)$}

\subsubsection{General colored links}
Link invariants can be computed using the operator invariants of colored braids by presenting links as braid closures.  In general, the link invariant would be a Markov trace of the associated matrix.  But for the discrete gauge theory representations, the trace is simply the linear algebra trace as shown in \cite[Sect. 5]{AC92}.  For a particular link $L$ as the braid closure $\hat{b}$, the colors cannot be arbitrary as the same component of the link $L$ has to be colored by the same irrep.

\subsubsection{Invariants of colored Whitehead links}

The colored invariants of the Whitehead link are calculated by computers, and assembled into the $W$-matrices in Appendix \ref{AppendixB}. Later, special attention is paid to the $BA$ block of the $W$-matrix, i.e. the sub-matrix of $W$ coming from labeling one component by anyons of type A and the other by ones of type B. The $BA$ block has a closed formula as follows.
\begin{equation}
W^{(u)}_{B_{k,n}, A_{l,m}} = 55 (-1)^{ lm [k^2]_5}
\cdot\left(\theta_{A_{l,m}}^{\frac{1}{2}[k^2]_5} \cdot\theta^{(u)}_{B_{k,n}}\right)^{-1}
\end{equation}
where $[x]_y = x\mod y$ and our convention for taking roots is $(e^{\frac{2\pi i s}{N}})^{t} = e^{\frac{2\pi i st}{N}}$. $k,n,l,m,u$ take values in $k\in\{1,2,3,4\}$, $n\in\{0,1,2,3,4\}$, $l\in\{1,2\}$, $m\in\{0,\ldots,10\}$, and $u\in\{0,1,2,3,4\}$.

\subsubsection{Links colored by the same anyon}

The counting of group homomorphisms for the DW invariants of links is generalized to links colored by a single irrep of MS-MCs using quandles below.

\begin{definition}\label{quandle}
A \emph{quandle} $X$ is a set with a binary operation $(x, y) \mapsto x\triangleright y$ satisfying:

\begin{enumerate}
    \item for any $x\in X$, $x\triangleright x=x$,
    \item  for any $x, y \in X$, there is a unique $z \in X$ such that $z \triangleright  x = y$,
    \item  for any $x, y, z \in X$, we have $x \triangleright (y\triangleright z)= (x\triangleright y)\triangleright (x \triangleright z)$.
\end{enumerate}
\end{definition}

The prototypical example of a quandle is a conjugacy class of a group $X\subset G$ with operation $g \triangleright h=g^{-1}hg$.

\begin{example}\label{exmple of quandle}
\begin{enumerate}
    \item A left $\Z[t,t^{-1}]$-module is a quandle with $x\triangleright y= (1-t)x+ty$. This quandle is called an Alexander quandle.
    \item Given $G(q,p;n)=\Z_q\rtimes _n\Z_p=\langle a\rangle\rtimes_n \langle b\rangle$, the quandles associated to the conjugacy classes of $b^k,k=0,1,...,p-1$ are Alexander quandles: 
\begin{align*}
    X_k=\Z/q\Z,&&  a\triangleright b:=(1-n^k)a+n^kb 
\end{align*}
for  $a, b \in \Z/q\Z$.
\end{enumerate}
\end{example}

A quandle defines a set-theoretical solution of the braid equation,
\begin{align*}
    c:X\times X\to X\times X,&&  c(x,y)=(x\triangleright y,x).
\end{align*}
Given $q\in \C^{\times}$, the linear map $c^q:\C[X]\otimes \C[X]\to \C[X]\otimes \C[X]$, given by \[c^q(|x\rangle |y\rangle)=q|x\triangleright y\rangle | x\rangle\]is a solution of the braid equation. The associated representation of the braid group will be denoted by $\rho_n^q:\cB_n\to GL(\C[X]^{\otimes n})$.

\begin{proposition}\label{prop:represention braid}
The representation of $\cB_n$ associated to the simple object $V(b^k,\pi_k^s)$ is the representations $\rho_n^q$ defined by the solution of the braid equation associated to the quandle $X_k$ and the scalar $q=e^{\frac{2\pi i}{p^2}(sp+u k)k}$.
\end{proposition}
\begin{proof}
The representation of $\cB_n$ on
\[V^{* n}:= (\cdots (V\otimes V)\otimes \cdots \otimes V)\otimes V, \] 
is given by the  isomorphisms
\begin{align}
\si' = (a^{-1}_{V^{*(i-1)},V,V}&\otimes \id_{V^{*(n-(i+2))}})\circ \label{generadores}\\
(\id_{V^{i-1}}&\otimes c_{V,V}\otimes \id_{V^{*(n-(i+2))}})\circ \notag\\
 (&a_{V^{*(i-1)},V,V}\otimes \id_{V^{*(n-(i+2))}}), \notag
\end{align}where $a_{V,W,Z}$ denotes the associativity constraints. 

The associativity constraint is given by
\begin{align*}
    a_{V^{*(i-1)},V,V}(|x_1\rangle \cdots |x_n\rangle )= \omega^{-1}(g,b^k,b^k)|x_1\rangle \cdots |x_n\rangle,
\end{align*}where $g\in G(q,p;n)$.
Then, using \eqref{formula:braiding} we have that
\begin{align*}
    \si' (|x_1\rangle \cdots |x_n\rangle ) &=\omega^{-1}(g,b^k,b^k)\\
    &\times \omega(g,b^k,b^k)\\
    &\times \pi^s_k(b^k)|x_1\rangle \cdots |x_i\triangleright x_{i+1}\rangle | x_i\rangle|x_{i+2}\rangle\cdots |x_n\rangle\\
    &= \rho_n^q(\si)(|x_1\rangle \cdots |x_n\rangle ).
\end{align*}
\end{proof}


A \emph{coloring} of an oriented knot diagram by a quandle $X$ is a map $f: \A \to X$ from the set $\A$ of the over-arcs of the knot diagram to $X$ such that the image of the map satisfies the relation in Figure \ref{crossing} at each crossing:

\begin{figure}[ht]
\centering

$$\begin{tikzpicture}
\poscross{\phantom{b}a\phantom{b}}{a \vartriangleright b}{b}{a}
\end{tikzpicture}  \hspace{50pt}\begin{tikzpicture} 
\negcross{b}{\phantom{b}a\phantom{b}}{a}{a \vartriangleright b}
\end{tikzpicture}$$

\caption{The coloring rule}
\label{crossing}
\end{figure}

The number of colorings of an oriented knot diagram $L$ by a quandle $X$ is an knot invariant and will be denoted by $C_X(L)$.

\begin{theorem}
Let $\sigma_{i_1}^{\epsilon_1}\sigma_{i_2}^{\epsilon_2}\cdots \sigma_{i_h}^{\epsilon_h} \in \cB_n$, and $L=\widehat{\sigma}$ its closure. The invariant of the oriented framed link $L$ colored by the simple object $V(b^k,\pi_k^s)$ is
\begin{equation}
    q^{Wr(L)}C_{X_k}(L),
\end{equation}
where $Wr(L)$ is the writhe of $L$, $X_k$ is the Alexander quandle constructed in Example \ref{exmple of quandle} and $q=e^{\frac{2\pi i}{p^2}(sp+uk )k}$, which is the twist $\theta_{([b^k],\pi_k^s)}$ of the anyon $([b^k],\pi_k^s)$.
\end{theorem}
\begin{proof}
Since the quantum trace of the category of representations of $D^{\omega_u}(G)$ is the usual trace (see \cite[Sect. 5]{AC92}), it follows from Proposition \ref{prop:represention braid} that the link invariant associated to the simple object $V(b^k,\pi_k^s)$ is the scalar $\operatorname{tr}(\rho_n^q(\sigma_{i_1}^{\epsilon_1}\sigma_{i_2}^{\epsilon_2}\cdots \sigma_{i_h}^{\epsilon_h}))$. Thus,
\begin{align*}
    \rho_n^q(\sigma_{i_1}^{\epsilon_1}\sigma_{i_2}^{\epsilon_2}\cdots \sigma_{i_h}^{\epsilon_h})&=\rho_n^q(\sigma_{i_1})^{\epsilon_1}\rho_n^q(\sigma_{i_2})^{\epsilon_2}\cdots \rho_n^q(\sigma_{i_h})^{\epsilon_h}\\
    &=(q^{\sum_i \epsilon_i})\rho_n(\sigma_{i_1})^{\epsilon_1}\rho_n(\sigma_{i_2})^{\epsilon_2}\cdots \rho_n(\sigma_{i_h})^{\epsilon_h}\\
    &= (q^{\sum_i \epsilon_i})\rho_n (\sigma_{i_1}^{\epsilon_1}\sigma_{i_2}^{\epsilon_2}\cdots \sigma_{i_h}^{\epsilon_h}).
\end{align*}Thus, 
\[\operatorname{tr}(\rho_n^q(\sigma_{i_1}^{\epsilon_1}\sigma_{i_2}^{\epsilon_2}\cdots \sigma_{i_h}^{\epsilon_h}))=(q^{\sum_i \epsilon_i})\operatorname{tr}(\rho_n(\sigma_{i_1}^{\epsilon_1}\sigma_{i_2}^{\epsilon_2}\cdots \sigma_{i_h}^{\epsilon_h})).\]
The writhe of $L$ is exactly $\sum_i\epsilon_i$ and it is well known that for a quandle $X$ $$\operatorname{tr}(\rho_n(\sigma_{i_1}^{\epsilon_1}\sigma_{i_2}^{\epsilon_2}\cdots \sigma_{i_h}^{\epsilon_h}))$$ is $C_X(L)$, the number of colorings, (see \cite{grana}).
\end{proof}

\begin{example}
The Borromean rings $L_B$ is the closure of $(\sigma_2\sigma_1^{-1})^3\in \cB_3$.  The invariant $C_{X_k}(L_B)\in \mathbb{N}$ and does not depend on $u, s\in \{0,1\ldots, p-1\}$. The same is true for the figure-eight knot, which is the closure of the
braid $(\sigma_1\sigma_2^{-1})^2$.
\end{example}

\subsection{$W$-matrix of MS-MC $\mcZ(\vgou)$ for $G(11,5;4)$}

There are three types of anyons in the MS-MC $\mcZ(\vgou)$ for $G(11,5;4)$: the $7$ anyons of type $I$, the $22$ anyons of type $A$, and the $20$ anyons of type $B$.  In order to show that the $W$-matrix can detect each of the five $\mcZ(\vgou)$, it is sufficient to show that there are no permutations mapping the $W$-matrices onto one another. This is because any braided tensor equivalence between the categories would restrict to a permutation of the anyons, and hence induce a permutation of any matrix of invariants. 

Suppose $\phi: \mathcal{Z}(\textbf{Vec}_G^{\omega^u}) \to \mathcal{Z}(\textbf{Vec}_G^{\omega^{u'}})$ were a braided tensor equivalence of categories. Then the modular data obeys $S_{\phi(a)\phi(b)}=S_{ab}$ and $T_{\phi(a)}=T_a$. That is, the modular data is invariant under the induced permutation on the anyons. An analysis of the modular data shows that there are three distinct sets up to permutation: $u=0$, $u=1,4$, and $u=2,3$. Thus it suffices to show that no permutations mapping the modular data of $u=1$ onto $u=4$ extend to permutations of the $W$-matrices between the same values of $u$ (and similarly for $u=2,3$). Then one can confirm that there are really five braided tensor equivalence classes of categories $\mcZ(\vgou)$, reproducing the result of Mignard-Schauenberg \cite{Mignard-Shauenburg}.

A simple inspection of the explicit $T$ matrix Table \ref{explicitT} shows that to identify the $T$ matrix for $u=1$ with $u=4$, $B_{1,0}$ of $u=1$ has to be sent to $B_{2,1}$ or $B_{3,1}$ of $u=4$.  We will show that that this permutation of anyons is not compatible with the $W$ matrix.  Therefore, there are no permutations of anyons between $u=1$ and $u=4$.  Actually something stronger is true; there are no permutations at all to match up the $W$-matrix for $u=1$ and $u=4$.  Since it is hard to show this fact analytically, an alternative proof using both $T$ and $W$ is given below.

\subsubsection{Proof that no permutation of the modular data extends to the $W$ matrices}\label{nomatch}
We prove that permutations taking $T^{(u=1)}$ to $T^{(u=4)}$ cannot simultaneously take the $W^{(u=1)}_{B,A}$ block to the $W^{(u=4)}_{B,A}$ block.

From Table (\ref{explicitT}), the topological spins $\theta=\xi_{25}^s$ can be calculated in the same order of anyons as in the table for $$u=1, s\in \{1,6,11,16,21;4,14,24,9,19;9,24,14,4,19;16,11,6,1,21\}$$ and for $$u=4, s\in  \{4,9,14,19,24;16,1,11,21,6;11,1,16,6,21;14,9,4,24,19\}.$$  Therefore, to match up the $T$ matrix, $B^{(u=1)}_{1,0}$ has to be sent to either $B^{(u=4)}_{2,1}$ or $B^{(u=4)}_{3,1}$.  Similarly, such a permutation has to send $A^{(u=1)}_{1,4}$ to $A^{(u=4)}_{1,4}$ or $A^{(u=4)}_{2,2}$.

\begin{proposition}
\begin{equation}\label{WABblock}
W^{(u)}_{B _{k,n}, A_{l,m}} = 55\chi_k(l,m) \left( \theta_{A_{l,m}}^{\eta(k)} \theta^{(u)}_{B_{k,n}}
\right)^{-1}
\end{equation}
$\chi_k(l,m) = (-1)^{lm[k^2]_5} =-1$ if $k=1$ or $k=4$ and $l=1$ and $m \in \{1,3,5,7,9\}$.  Otherwise,
$\chi_k(l,m)=+1$. $\eta(k) =\frac{1}{2}[k^2]_5 =\frac{1}{2}$ if $k=1$ or $k=4$ and $\eta(k) = 2$ if $k=2$ or $k=3$.
\end{proposition}
Our convention for taking roots is $(e^{\frac{2\pi i s}{N}})^{t} = e^{\frac{2\pi i st}{N}}$. All full $W^{(u)},u=0,1,2,3,4,$ matrices are known (see Appendix \ref{AppendixB}), but only those entries are provided because the remaining ones are not used anywhere.

From the modular $T$ matrix Table (\ref{explicitT}), it is found that between $u=1$ and $u=4$,
\begin{align}
	B_{1,0}^{(1)} \to
	\begin{cases}
		B_{2,1}^{(4)}\\
		B_{3,1}^{(4)}
	\end{cases} \quad \text{and} \quad
	A_{1,4}^{(1)} \to
	\begin{cases}
		A_{1,4}^{(4)}\\
		A_{2,2}^{(4)}
	\end{cases}
	.
\end{align}

Now consider the permutation of $B_{1,0}^{(1)}$ to $B_{2,1}^{(4)}$ or $B_{3,1}^{(4)}$ in the $W^{(u)}_{B,A}$ block. This will induce a permutation of $A_{1,4}^{(1)}$ to some $A_{l,m}^{(4)}$. That is, we wish to find $l,m$ such that
\begin{align}\label{W1BA}
	W^{(1)}_{B_{1,0}, A_{1,4}} & = \begin{cases}
		W^{(4)}_{B_{2,1}, A_{l,m}} \\
		W^{(4)}_{B_{3,1}, A_{l,m}}
	\end{cases}.
\end{align}

Using Eq. (\ref{WABblock}), we obtain:
\begin{align}
	W^{(1)}_{B_{1,0}, A_{1,4}} &= 55\left(\theta_{A_{1,4}}^{\frac{1}{2}} \theta^{(1)}_{B_{1,0}} \right)^{-1}  = 
	\begin{cases}
		W^{(4)}_{B_{2,1}, A_{l,m}} &=55 \left(\theta_{A_{l,m}}^2 \theta^{(4)}_{B_{2,1}}\right)^{-1} \\
		W^{(4)}_{B_{3,1}, A_{l,m}} &=55 \left(\theta_{A_{l,m}}^2 \theta^{(4)}_{B_{3,1}} \right)^{-1}
	\end{cases}.
\end{align}

The $\theta_B$'s cancel out since they come from the permutations of $T$ and the above relation leads to the possible solutions for $l,m$.
\begin{align}
	\theta_{A_{1,4}}^{-\frac{1}{2}} = \exp\left( \frac{-2\pi i}{11} 1\cdot 4 \cdot \frac{1}{2}\right) &=
	\exp\left(\frac{-2\pi i}{11} l \cdot m \cdot 2\right) = \theta_{A_{l,m}}^{-2} .
\end{align}
Therefore, $W_{B,A}^{(u)}$ requires
\begin{equation}\boxed{
	A_{1,4}^{(1)} \to A_{1,1}^{(4)} \textrm{ or } A_{2,6}^{(4)}
}\end{equation}

But this does not match up with the permutations that $T$ would allow:
\begin{equation}\label{P1to2}
A_{1,4}^{(1)} \to A_{1,4}^{(4)} \textrm{ or } A_{2,2}^{(4)}.
\end{equation}

The contradiction proves that there are no permutations of anyons of $u=1$ and $u=4$ that would match up both the $T$ and $W$ matrices. A similar result holds for permutations between $u=2$ and $u=3$. Thus the $W$-matrices in addition to the modular data $(S,T,W)$ distinguish the $\mcZ(\vgou)$. Alternatively one can check from the $W$-matrices directly that they cannot be related by any permutation of the anyons.

\subsection{MS-MCs $\mcZ(\vgou)$ as gauging $\Z_p$ symmetries of $D(\Z_{q})$}

The MS-MCs can also be obtained as gauging the $G(q,p;n)$ symmetry of $\textrm{Vec}$.  By the sequentially gauging lemma in \cite{cui16}, they are also the gaugings of $\Z_p$ symmetries of $\mcZ(\Vec_{\Z_q})$, denoted as $\mathcal{D}(\Z_q)$ in the physics literature.

For $q=11,p=5$, the anyons of $\mathcal{D}(\Z_q)$ are labeled as $a=(a_1,a_2)$ and the $F$ matrices are $F^{abc}=1$ if admissible and $0$ otherwise.  The $R$ symbols are $R^{ab}=e^{\frac{2\pi i}{11}a_1b_2}$.  The $\Z_5$ symmetry acts on anyons as $n.(a_1,a_2)=([9^na_1]_{11},[5^na_2]_{11}), n\in \Z_5$.  In the $G$-crossed extension, there is one defect in each sector, denoted as $X_m, m=1,2,3,4$.  The fusion rules of defects are:
$$X_mX_{5-m}=\sum_{a\in \mathcal{D}(\Z_{11})}a, aX_m=X_m, X_mX_n=11X_{[m+n]_{11}}\; \textrm{if}\; m+n\neq 0 \mod 5 .$$

The $\Z_5$ action fixes all defects, $\uno$ of $\mathcal{D}(\Z_q)$, and has $24$ orbits of size $5$.  Therefore, in the gauged theory, there are $4\times 5=20$ anyons of quantum dimension=$11$ from the defects, $24$ anyons of quantum dimensions=$5$ from the $24$ orbits, and $5$ abelian anyons.  Another way to see that the MS-MCs are different would be to show that the $5$ $G$-crossed extensions are also different.  We will leave this approach to the future.

\section{Systematic search}

It would be interesting to search through all knots and links systematically to find those links whose invariants are beyond the modular data.   We have started such a program and will leave the results for a future publication \cite{SSDT}.  After finishing this manuscript, we learned of an independent project to search for link invariants that are not determined by modular data \cite{PSprivate}.

Such a systematic search will shed light on the candidates of a complete set of invariants for modular categories.  A good test for whether or not a set $\mathcal{L}_C$ of link invariants would be complete is whether or not the eigenvalues of the punctured $S$-matrices and the triple $(S,T,W)$ are determined by the set $\mathcal{L}_C$.


\appendix
\section{$T$-matrices}
\begin{equation}\label{explicitT}
\scalemath{0.90}{
	\begin{array}{c|cc}
	\textrm{Label} & d & \theta \\ \hline \hline
	I_0 & 1 & 1\\
	I_1 & 1 & 1\\
	I_2 & 1 & 1\\
	I_3 & 1 & 1\\
	I_4 & 1 & 1\\
	I_5 & 5 & 1\\
	I_6 & 5 & 1\\ \hline
	A_{1,0} & 5 & 1\\
	A_{1,1} & 5 & \exp(\frac{i2\pi}{11})\\
	A_{1,2} & 5 & \exp(\frac{i2\pi}{11}2)\\
	A_{1,3} & 5 & \exp(\frac{i2\pi}{11}3)\\
	A_{1,4} & 5 & \exp(\frac{i2\pi}{11}4)\\
	A_{1,5} & 5 & \exp(\frac{i2\pi}{11}5)\\
	A_{1,6} & 5 & \exp(\frac{i2\pi}{11}6)\\
	A_{1,7} & 5 & \exp(\frac{i2\pi}{11}7)\\
	A_{1,8} & 5 & \exp(\frac{i2\pi}{11}8)\\
	A_{1,9} & 5 & \exp(\frac{i2\pi}{11}9)\\
	A_{1,10} & 5 & \exp(\frac{i2\pi}{11}10)\\ \hline
	A_{2,0} & 5 & 1\\
	A_{2,1} & 5 & \exp(\frac{i4\pi}{11})\\
	A_{2,2} & 5 & \exp(\frac{i4\pi}{11}2)\\
	A_{2,3} & 5 & \exp(\frac{i4\pi}{11}3)\\
	A_{2,4} & 5 & \exp(\frac{i4\pi}{11}4)\\
	A_{2,5} & 5 & \exp(\frac{i4\pi}{11}5)\\
	A_{2,6} & 5 & \exp(\frac{i4\pi}{11}6)\\
	A_{2,7} & 5 & \exp(\frac{i4\pi}{11}7)\\
	A_{2,8} & 5 & \exp(\frac{i4\pi}{11}8)\\
	A_{2,9} & 5 & \exp(\frac{i4\pi}{11}9)\\
	A_{2,10} & 5 & \exp(\frac{i4\pi}{11}10)\\ \hline
	B_{1,0} & 11 & \exp(\frac{i2\pi}{25} 1^2 u) \\
	B_{1,1} & 11 & \exp(\frac{i2\pi}{25}(5\cdot 1\cdot 1+1^2 u)) \\
	B_{1,2} & 11 & \exp(\frac{i2\pi}{25}(5\cdot 1\cdot 2+1^2 u)) \\
	B_{1,3} & 11 & \exp(\frac{i2\pi}{25}(5\cdot 1\cdot3+1^2 u)) \\
	B_{1,4} & 11 & \exp(\frac{i2\pi}{25}(5\cdot 1\cdot4+1^2 u)) \\ \hline
	B_{2,0} & 11 & \exp(\frac{i2\pi}{25} 2^2 u) \\
	B_{2,1} & 11 & \exp(\frac{i2\pi}{25}(5\cdot 2\cdot1+2^2 u)) \\
	B_{2,2} & 11 & \exp(\frac{i2\pi}{25}(5\cdot 2\cdot2+2^2 u)) \\
	B_{2,3} & 11 & \exp(\frac{i2\pi}{25}(5\cdot 2\cdot3+2^2 u)) \\
	B_{2,4} & 11 & \exp(\frac{i2\pi}{25}(5\cdot 2\cdot4+2^2 u)) \\ \hline
	B_{3,0} & 11 & \exp(\frac{i2\pi}{25} 3^2 u) \\
	B_{3,1} & 11 & \exp(\frac{i2\pi}{25}(5\cdot 3\cdot1+3^2 u)) \\
	B_{3,2} & 11 & \exp(\frac{i2\pi}{25}(5\cdot 3\cdot2+3^2 u)) \\
	B_{3,3} & 11 & \exp(\frac{i2\pi}{25}(5\cdot 3\cdot3+3^2 u)) \\
	B_{3,4} & 11 & \exp(\frac{i2\pi}{25}(5\cdot 3\cdot4+3^2 u)) \\ \hline
	B_{4,0} & 11 & \exp(\frac{i2\pi}{25} 4^2 u) \\
	B_{4,1} & 11 & \exp(\frac{i2\pi}{25}(5\cdot 4\cdot1+4^2 u)) \\
	B_{4,2} & 11 & \exp(\frac{i2\pi}{25}(5\cdot 4\cdot2+4^2 u)) \\
	B_{4,3} & 11 & \exp(\frac{i2\pi}{25}(5\cdot 4\cdot3+4^2 u)) \\
	B_{4,4} & 11 & \exp(\frac{i2\pi}{25}(5\cdot 4\cdot4+4^2 u)) \\ 
	\end{array}}
\end{equation}

\section{W-matrices}\label{AppendixB}
The Whitehead link is presented as the closure of the $3$-strand braid $b_7$ with braid word $\sigma_2^{-2} \sigma_1^{-1} \sigma_2^2 \sigma_1^2$ or $b_5=\sigma_2^{-2} \sigma_1 \sigma_2^{-1} \sigma_1$.  The braid $b_5$ represent our Whitehead link with an extra twist, and $b_7$ its mirror image with an extra twist.

The $W$-matrices are too large to be included in here, so instead we post them on the following two websites.

\begin{itemize}
    \item  \url{http://web.physics.ucsb.edu/~adtran/W.html}
\item \url{http://web.math.ucsb.edu/~cdelaney/WMatrices.html}
\end{itemize}


\end{document}